\documentclass[10pt,reqno]{amsart}
\usepackage{bbm}
\usepackage{mathrsfs}
\usepackage{amsfonts}
\usepackage[dvipsnames,usenames]{color}
\usepackage{graphicx,latexsym,bm,amsmath,amssymb,verbatim,multicol,lscape}
\newtheorem{thm}{Theorem} [section]
\newtheorem{cor}[thm]{Corollary}
\newtheorem{lem}[thm]{Lemma}

\theoremstyle{definition}

\theoremstyle{remark}
\newtheorem{rem}[thm]{Remark}

\numberwithin{equation}{section}

\newcommand{\N}{\mathbb N}

\DeclareMathOperator{\ord}{ord}

\DeclareMathOperator{\supp}{supp}

\DeclareMathOperator{\st}{st}

\begin{document}
\title[The zero-sum invariant on weak-regular sequences]{The zero-sum invariant on weak-regular sequences}

\date{}
\author[W.K. Yang]{Wenkai Yang}
\address{Center for Combinatorics, Nankai University, Tianjin 300071, P.R. China}
\email{9820250116@nankai.edu.cn}
\renewcommand{\keywordsname}{Keywords}

\begin{abstract}
For a nontrivial finite abelian group $G$, let $\mathsf N(G)$ be the smallest
integer $l$ such that every sequence $S$ over $G$ of length at least $l$ in
which each element $g$ occurs at most $\ord(g)$ times has a nonempty zero-sum
subsequence containing all occurrences of some term of $S$. Gao et al.\ introduced $\mathsf N(G)$
to study the possible values of the universal zero-sum invariant.
Exact values of $\mathsf N(G)$ are known for only a few families of groups.
We show that $r(m-1)+1$ is a sharp lower bound for $\mathsf N(G)$, where
$r$ and $m$ are the rank and exponent of $G$, respectively.
Using $c_0(G)$ from the classical theory of additive bases, we prove that
this lower bound is attained for $G=C_p\oplus C_{np}$ whenever $n\ge2$
and $p\ge3$ is the smallest prime divisor of $np$. For these groups,
$\mathsf N(G)=2np-1$, and we characterize the extremal weak-regular sequences.
\end{abstract}

\maketitle

\section{Introduction and main results}

Many zero-sum invariants are defined as the smallest length at which every sequence
over a nontrivial finite abelian group $G$ has a nonempty zero-sum subsequence
of a prescribed type. For the Davenport constant $\mathsf D(G)$, any nonempty
zero-sum subsequence suffices.

In 1963 Rogers~\cite{[R1963]} studied the invariant now known as the
Davenport constant, and H.~Davenport later brought the problem to wider
attention~\cite{[W2026]}. The constant arose in the study of ideal class
groups of number fields and has applications to the study of sets of
lengths in Krull monoids~\cite{[GG2006],[GG2009],[GS2021]}.
Exact values and bounds for $\mathsf D(G)$ have been widely
studied~\cite{[G2000],[GG2006],[GGWZ2026],[GGZ2026],[GH2006],[GS1992],[HZ2026],[L2026],[L2020],[S2011],[W2026]}.

In 2018 Gao et al.~\cite{[GLPW2018]} introduced the universal zero-sum
invariant $\mathsf d_{\Omega}(G)$ to study zero-sum invariants in a common
setting. Wang subsequently studied families of
zero-sum sequences representing $\mathsf D(G)$ and weighted zero sums
over infinite abelian groups~\cite{[W2025],[W2026II]}.

Regularity plays a key role in proving the existence of zero-sum subsequences
with prescribed properties. To study the universal zero-sum invariant for
a wider class of sequences, Gao et al.~\cite{[GLPW2018]} introduced
weak-regular sequences. For a sequence $S$ over $G$, let $\mathsf v_g(S)$
denote the multiplicity of $g$ in $S$. The sequence $S$ is weak-regular
if $\mathsf v_g(S)\le\ord(g)$ for every $g\in G$.

Gao et al.~\cite{[GHHLYZ2022-1],[GHHLYZ2022-2]} introduced $\mathsf N(G)$
to study which integers occur as $\mathsf d_{\Omega}(G)$ when $\Omega$
ranges over families of weak-regular zero-sum sequences. By definition,
$\mathsf N(G)$ is the smallest integer $l$ such that every weak-regular
sequence $S$ over $G$ of length at least $l$ has a nonempty zero-sum
subsequence $T$ with $\mathsf v_g(T)=\mathsf v_g(S)$ for some term $g$ of $S$.
Exact values are known for cyclic groups of prime power order, for
$C_p^2$ with $p$ prime, and for elementary abelian groups of exponent
two or three~\cite[Theorem~1.3]{[GHHLYZ2022-2]}. For general $n\ge3$,
Jiang and the author narrowed $\mathsf N(C_n)$ to two consecutive
integers~\cite[Theorem~1.1]{[JY2024]}.

By the structure theorem for finite abelian groups,
$G\cong C_{n_1}\oplus\cdots\oplus C_{n_r}$ for unique integers
$1<n_1\mid\cdots\mid n_r$, where $r$ is the rank of $G$ and
$n_r=\exp(G)$. Very recently, Wang obtained bounds for $\mathsf D(G)$
in terms of $rn_r$~\cite[Remark~6.4 and Theorem~6.6]{[W2026]}.
By the definition of $\mathsf D(G)$, there is a zero-sum free sequence of
length $\mathsf D(G)-1$. Each $g\in G$ occurs fewer than $\ord(g)$ times
in this sequence, so it is weak-regular. Hence $\mathsf N(G)\ge\mathsf D(G)$.

In this paper we prove the sharp lower bound $\mathsf N(G)\ge r(n_r-1)+1$.
The other term in Theorem~\ref{thm1.4} is the direct-sum bound of
Gao et al.~\cite[Theorem~1.1(2)]{[GHHLYZ2022-2]}.

\begin{thm}\label{thm1.4}
Let $G=C_{n_1}\oplus\cdots\oplus C_{n_r}$, where
$1<n_1\mid\cdots\mid n_r$. Then
\[
\mathsf{N}(G)\ge
\max\left\{r(n_r-1)+1,\ \sum_{i=1}^r\bigl(\mathsf{N}(C_{n_i})-1\bigr)+1\right\}.
\]
\end{thm}

For $G=C_p\oplus C_{np}$, it follows from Theorem~\ref{thm1.4} that
$\mathsf N(G)\ge 2np-1$. Under the following hypotheses we prove equality
and characterize the extremal weak-regular sequences.

\begin{thm}\label{thm1.5}
Let $G=C_p\oplus C_{np}$, where $n\ge2$ and $p\ge3$ is the
smallest prime dividing $|G|$. Then
\[
 \mathsf N(G)=2np-1.
\]
Moreover, a weak-regular sequence $S$ over $G$ of length $2np-2$ satisfies
$\mathsf v_x(T)<\mathsf v_x(S)$ for every nonempty zero-sum subsequence
$T\mid S$ and every term $x$ of $S$ if and only if
\[
 S=g^{np-1}h^{np-1},\qquad
 \ord(g)=\ord(h)=np,\qquad \langle g\rangle\ne\langle h\rangle.
\]
\end{thm}

\begin{rem}
The bound $\mathsf N(G)\ge r(n_r-1)+1$ in Theorem~\ref{thm1.4} is sharp.
For every group $G$ in Theorem~\ref{thm1.5}, we have $r=2$ and $n_r=np$,
so $r(n_r-1)+1=2np-1=\mathsf N(G)$.
\end{rem}

The proof of Theorem~\ref{thm1.5} combines bounds for $c_0(G)$ with
estimates for the number of distinct subsequence sums. Determining $c_0(G)$ for
regular sequences extends the classical additive basis problem for sets
posed by Erd\H{o}s and Heilbronn~\cite{[EH1964]} in 1964. For results on additive bases in finite
abelian groups, see~\cite{[GHQQZ2015],[GLQW2022],[GQZ2020],[P1987I],[P1987II],[QH2017],[QL2022],[QLWZ2025]}.

The paper is organized as follows. In Section~3 we prove an upper bound for $c_0(C_q\oplus C_{nq})$ when $q$ is
prime, $n\ge2$ and $nq$ is odd. Section~4 proves Theorems~\ref{thm1.4}
and~\ref{thm1.5}. In Section~5 we determine which of the two possible
values of $\mathsf N(C_n)$ occurs for each of four infinite families.
Together with \cite[Lemma~5.1]{[JY2024]}, these results yield the formula
for $\mathsf N(C_n)$ stated in Corollary~\ref{cor:cyclic-finite-range}.

\section{Notation and preliminaries}

Our notation for sequences follows the standard conventions of zero-sum
theory as in \cite{[GG2006],[GH2006]}.
A sequence over $G$ is finite and unordered, and repetitions are allowed.
A sequence $S=g_1\cdots g_l$ has length $|S|=l$, sum
$\sigma(S)=g_1+\cdots+g_l$, and support
$\supp(S)=\{g\in G:\mathsf v_g(S)>0\}$.
The relation $T\mid S$ means that $T$ is a subsequence of $S$. In this case, $ST^{-1}$ is the unique sequence $U$ such that $S=TU$.
For the empty sequence $1$, we set $\sigma(1)=0$.
Subsequences $T_1,\ldots,T_l$ of $S$ are disjoint if $T_1\cdots T_l\mid S$.
A homomorphism $\varphi:G\to G_1$ induces the sequence
$\varphi(S)=\varphi(g_1)\cdots\varphi(g_l)$ over $G_1$.

For a subgroup $H\le G$, let $S_H$ denote the subsequence of all terms
of $S$ in $H$. The sequence $S$ is regular if $|S_H|\le |H|-1$ for
every proper subgroup $H<G$. Every zero-sum free sequence is regular,
since any $|H|$ terms in $H$ have a nonempty zero-sum subsequence.
For noncyclic $G$, every regular sequence $S$ is weak-regular, since
$\langle g\rangle<G$ and $\mathsf v_g(S)\le |S_{\langle g\rangle}|
\le\ord(g)-1$ for every $g\in G$.

Let $\Sigma(S)=\{\sigma(T):T\mid S,\ T\ne1\}$.
Let $\Sigma_0(S)=\{\sigma(T):T\mid S\}=\Sigma(S)\cup\{0\}$.
A sequence $S$ is an additive basis of $G$ if $\Sigma(S)=G$.
Let $c_0(G)$ be the smallest integer $l$ such that every regular sequence
over $G$ of length at least $l$ is an additive basis of $G$.

A nonempty sequence is a minimal zero-sum sequence if it has sum $0$
and no proper nonempty zero-sum subsequence.
Let $\N$ denote the positive integers. We use the usual floor and
ceiling functions $\lfloor x\rfloor$ and $\lceil x\rceil$.

The stabilizer of a subset $A$ of $G$ is $\st(A)=\{g\in G:g+A=A\}$.
We use Kneser's theorem in the form given in \cite[Lemma~2.1]{[GHQQZ2015]}.

\begin{lem}\label{lem2.1}
Let $A_1,\ldots,A_n$ be nonempty finite subsets of an abelian group $G$,
and let $H=\st(A_1+\cdots+A_n)$. Then
\[
 |A_1+\cdots+A_n|\ge |A_1+H|+\cdots+|A_n+H|-(n-1)|H|.
\]
\end{lem}

The next lemma follows by induction from the Cauchy--Davenport theorem \cite[Theorem~2.2]{[N1996]}.

\begin{lem}\label{lem2.2}
Let $A_1,\ldots,A_n$ be nonempty subsets of $C_p$, where $p$ is prime. Then
\[
 |A_1+\cdots+A_n|\ge\min\{p,\,|A_1|+\cdots+|A_n|-(n-1)\}.
\]
\end{lem}

Let $p$ be the smallest prime divisor of $|G|$, and let
\[
m(G)=
\begin{cases}
|G|,&G\text{ is cyclic},\\
2p-1,&G=C_p^2,\\
np+2p-3,&G=C_p\oplus C_{np},\ n\ge2,\\
|G|/p+p-2,&\text{otherwise}.
\end{cases}
\]

The lower and upper bounds for $c_0(G)$ are given in
\cite[Lemma~2.3]{[GQZ2020]} and \cite[Lemma~2.4]{[GHQQZ2015]}, respectively.

\begin{lem}\label{lem2.3}
For every nontrivial finite abelian group $G$, we have $m(G)\le c_0(G)\le |G|$.
\end{lem}

Qu and Han determined $c_0(G)$ for the following $p$-groups \cite[Theorem~1]{[QH2017]}.

\begin{lem}\label{lem2.4}
If $p$ is prime and $n\ge2$, then $c_0(C_p\oplus C_{p^n})=p^n+2p-3$.
\end{lem}

For a regular sequence that is not an additive basis, we use the stabilizer result
of Gao et al.\ \cite[Lemma~2.3(1)]{[GHQQZ2015]}.

\begin{lem}\label{lem5.2}
Let $p$ be the smallest prime divisor of the order of a finite abelian group $G$.
If $S$ is regular over $G$, $|S|\ge\max\{|G|/p+p-2,\mathsf D(G)\}$ and
$\Sigma(S)\ne G$, then $\st(\Sigma(S))=\{0\}$.
\end{lem}

For cyclic groups $C_n$ with $n\ge3$, we have
\begin{equation}\label{eq:N-cyclic-bound}
 2n-\lceil2\sqrt{n+1}\rceil\le\mathsf N(C_n)\le 2n-\lceil2\sqrt{n+1}\rceil+1.
\end{equation}
The upper bound is due to Gao et al.~\cite[Theorem~1.1(1)]{[GHHLYZ2022-2]}
and the lower bound to Jiang and the author~\cite[Theorem~1.1]{[JY2024]}.
For $n=2$ we have $\mathsf N(C_2)=2$. For elementary abelian groups of rank two,
Gao et al.\ \cite[Theorem~1.3(3)]{[GHHLYZ2022-2]} proved that
\begin{equation}\label{eq:N-elementary-rank-two}
 \mathsf N(C_p^2)=4p-2\lceil2\sqrt{p+1}\rceil+1 \qquad(p\ge3\text{ prime}).
\end{equation}

We also use Hall's theorem \cite[Theorem~1]{[Hall1935]}:
finite sets $A_1,\ldots,A_m$ admit distinct representatives
$a_i\in A_i$ if and only if
\begin{equation}\label{eq:Hall}
 \left|\bigcup_{i\in I}A_i\right|\ge |I|\qquad\text{for every }I\subseteq\{1,\ldots,m\}.
\end{equation}

For $G=C_{n_1}\oplus\cdots\oplus C_{n_r}$ with $1<n_1\mid\cdots\mid n_r$,
Gao et al.~\cite[Theorem~1.1(2)]{[GHHLYZ2022-2]} proved that
\begin{equation}\label{eq:N-direct-sum}
 \mathsf N(G)\ge 1+\sum_{i=1}^r\bigl(\mathsf N(C_{n_i})-1\bigr).
\end{equation}
By \cite[Theorem~5.8.3]{[GH2006]}, for $1\le n_1\mid n_2$ we have
\begin{equation}\label{eq:D-rank-two}
 \mathsf D(C_{n_1}\oplus C_{n_2})=n_1+n_2-1.
\end{equation}

\begin{lem}\label{lem2.7}
Let $G=C_p\oplus C_{np}$, where $p\ge3$ is the smallest prime divisor
of $|G|$. For $1<m\mid n$, let $H\le pG$ have order $m$, and let $k=n/m$.
Then
\[
2np-1-\bigl(\mathsf N(H)-1+(m-1)kp\bigr)\ge
\begin{cases}
2kp-1,&m<n,\\
\mathsf N(C_p^2),&m=n.
\end{cases}
\]
\end{lem}
\begin{proof}
Since $H\le pG\cong C_n$ is cyclic and $m\ge p\ge3$,
we have $\mathsf N(H)\le2m-3$ by \eqref{eq:N-cyclic-bound}. The left side equals $kp(m+1)-\mathsf N(H)$
and is therefore at least $kp(m+1)-2m+3$.
For $m<n$, this exceeds $2kp-1$ by at least $(kp-2)(m-1)+2$.
So we assume that $m=n$. Substituting $k=1$ in $kp(m+1)-2m+3$, we obtain
$(p-2)n+p+3\ge p^2-p+3>4p-7\ge\mathsf N(C_p^2)$ by \eqref{eq:N-elementary-rank-two}.
The strict inequality follows from $p^2-5p+10=(p-2)(p-3)+4>0$.
\end{proof}

\begin{lem}\label{lem2.8}
Let $S$ be a weak-regular sequence over $G=C_n$, where $n\ge2$ and $|S|\ge n$.
Then either $\Sigma(S)=G$, or there is a zero-sum subsequence $T\mid S$
such that $\mathsf{v}_g(T)=\mathsf{v}_g(S)\ge1$ for some $g\mid S$.
\end{lem}
\begin{proof}
If $\mathsf v_g(S)=\ord(g)$ for some $g\mid S$, then $T=g^{\ord(g)}$ satisfies $\sigma(T)=0$ and $\mathsf v_g(T)=\mathsf v_g(S)$.
So we assume that $\mathsf{v}_g(S)<\ord(g)$ for every $g\mid S$. In particular, $0\notin\supp(S)$. We may also assume that $\Sigma(S)\ne G$.
Since $|S|\ge \mathsf D(G)$, $\Sigma(S)=\Sigma_0(S)$.
Let $H=\st(\Sigma_0(S))$ and $m=|H|$.
Let $S=g_1\cdots g_l$. Assume to the contrary that $m=1$, so that Lemma~\ref{lem2.1} applied to $\Sigma_0(S)=\sum_{j=1}^l\{0,g_j\}$ yields $|\Sigma_0(S)|\ge|S|+1>n=|G|$, a contradiction.
Thus $m>1$, and $m<n$ since $\Sigma_0(S)\ne G$.
Let $d=n/m\ge2$. Since $\Sigma_0(S)$ is a proper union of $H$-cosets, it has size at most $n-m$. For $g\mid S$, we have $|\{0,g\}+H|=2m$ if $g\notin H$ and $|\{0,g\}+H|=m$ if $g\in H$. Thus by Lemma~\ref{lem2.1},
\[
(|S|-|S_H|+1)m\le\bigl|\Sigma_0(S)\bigr|\le n-m,
\qquad
|S_H|\ge n-d+2\ge2m\ge\mathsf{N}(H),
\]
where $n-d+2-2m=(d-2)(m-1)\ge0$. The last bound follows from \eqref{eq:N-cyclic-bound} for $m\ge3$ and from $\mathsf N(C_2)=2$ for $m=2$.
By the definition of $\mathsf{N}(H)$, the weak-regular sequence $S_H$ has a nonempty zero-sum subsequence $T$ with $\mathsf v_g(T)=\mathsf v_g(S_H)$ for some $g\mid S_H$. Since $g\in H$, we have $\mathsf v_g(S_H)=\mathsf v_g(S)$.
\end{proof}

\begin{lem}\label{lem2.9}
Let $S,S'$ be sequences over a finite abelian group $G$, and let
$H\le G$ satisfy $H\subseteq\Sigma_0(S)$.
Suppose that $T_1\mid S'$ satisfies $\sigma(T_1)\in H$ and
$\mathsf{v}_g(T_1)=\mathsf{v}_g(S')\ge1$ for some
$g\in\supp(S')\setminus\supp(S)$.
Then $SS'$ has a zero-sum subsequence $T$ with
$\mathsf{v}_g(T)=\mathsf{v}_g(SS')$.
\end{lem}
\begin{proof}
Since $-\sigma(T_1)\in H\subseteq\Sigma_0(S)$, there exists $T_2\mid S$ with $\sigma(T_2)=-\sigma(T_1)$. Then $T=T_1T_2\mid SS'$ and $\sigma(T)=0$.
Since $g\notin\supp(S)$, we have $\mathsf v_g(T)=\mathsf v_g(T_1)=\mathsf v_g(S')=\mathsf v_g(SS')$.
\end{proof}

\begin{lem}\label{lem:c0-regular-sums}
Let $S$ be a regular sequence over a finite abelian group $G$. Then
\[
|\Sigma_0(S)|\ge \min\{|G|,|S|+1\}.
\]
In particular, $\Sigma_0(S)=G$ if $|S|\ge |G|-1$.
\end{lem}

\begin{proof}
Assume to the contrary that $|\Sigma_0(S)|<\min\{|G|,|S|+1\}$, and let $H=\st(\Sigma_0(S))$ and $m=|H|$.
Since $H\subseteq\Sigma_0(S)$, we have $H<G$ and $m\le |S|$. Let $S=g_1\cdots g_l$. By Lemma~\ref{lem2.1} applied to $\Sigma_0(S)=\sum_{j=1}^l\{0,g_j\}$ and the bound $|S_H|\le m-1$, we obtain
\[
|\Sigma_0(S)|\ge (|S|-|S_H|+1)m
\ge (|S|-m+2)m\ge |S|+1,
\]
a contradiction.
\end{proof}

\section{Upper bounds for additive bases}

\begin{lem}\label{lem:c0-short-block}
Let $K<H<G$, where $|K|\ge2$ and $H/K\cong C_q$ for a prime $q$.
Let $S$ be a sequence over $H\setminus K$ that is regular over $G$.
If $|S|\ge2q-1$, then there exists $T\mid S$ such that
$|T|\le q$ and $\sigma(T)\in K\setminus\{0\}$.
\end{lem}

\begin{proof}
Since $|S|\ge2q-1$, there exist subsequences $R,R'$ with $RR'\mid S$, $|R|=q$, and $|R'|=q-1$.
Every proper subgroup of $H$ is proper in $G$, so $R\mid S$ is regular over $H$.
By Lemma~\ref{lem:c0-regular-sums}, $|\Sigma_0(R)|\ge q+1$ since $|H|=q|K|\ge2q$.
There are only $q$ cosets of $K$ in $H$, so there exist $R_1,R_2\mid R$ such that $\sigma(R_1)-\sigma(R_2)\in K\setminus\{0\}$.
Let $\bar g=g+K$ for $g\in H$.
Since $g+K\ne K$ for every $g\mid R'$ and $|R'|=q-1$, we have $\Sigma_0(\overline{R'})=H/K$ by Lemma~\ref{lem2.2} in $H/K\cong C_q$.
Hence there exists $R_3\mid R'$ such that $\sigma(R_3)\in-\sigma(R_2)+K$.
Thus $\sigma(R_1R_3)$ and $\sigma(R_2R_3)$ are distinct elements of $K$.
Let $U\in\{R_1R_3,R_2R_3\}$ with $\sigma(U)\ne0$.
Since $\sigma(\bar U)=0$, there is a decomposition $U=U_1\cdots U_l$ such that
$\overline{U_i}$ is a minimal zero-sum sequence over $H/K$ for every $i$.
Since $\mathsf D(H/K)$ is the maximum length of a minimal zero-sum sequence over $H/K$,
we have $|U_i|=|\overline{U_i}|\le\mathsf D(H/K)=\mathsf D(C_q)=q$ for every $i$.
Moreover, $\sigma(U_i)+K=\sigma(\overline{U_i})=\bar 0$, so $\sigma(U_i)\in K$.
At least one $\sigma(U_i)\in K\setminus\{0\}$ since $\sum_{i=1}^l\sigma(U_i)=\sigma(U)\ne0$.
\end{proof}

\begin{lem}\label{lem:c0-decomposition}
Let $K\le G$ with $|K|\ge2$ and $G/K\cong C_q^2$ for a prime $q$.
If $S$ is regular over $G$ and
\[
|S|\ge q|K|+2q^2-q-1,
\]
then $S=S_KT_1\cdots T_kR$, where $k=|K|-|S_K|$,
$|T_i|\le q$, $\sigma(T_i)\in K\setminus\{0\}$ and $\Sigma_0(R)+K=G$.
\end{lem}

\begin{proof}
First, to select $k=|K|-|S_K|$ disjoint subsequences, let $\varphi:G\to G/K$ be the natural homomorphism. Since $S$ is regular, $|S_K|\le |K|-1$ and $k\ge1$.
There are $q+1$ subgroups $H$ with $K<H<G$, corresponding to the subgroups of order $q$ in $G/K\cong C_q^2$.
Each $g\in G\setminus K$ belongs to exactly one of them, since $g+K$ generates a subgroup of order $q$.
Let $S_1=SS_K^{-1}$. For each $i\in\{1,\ldots,k\}$, there exists a subgroup $H_i$ with $K<H_i<G$ such that $|(S_i)_{H_i}|=\max_{K<H<G}|(S_i)_H|$.
Since $|T_j|\le q$ for $1\le j<i$, we have
$|S_i|=|S_1|-\sum_{j=1}^{i-1}|T_j|\ge |S|-|S_K|-q(k-1)\ge2q^2-1$.
Thus $(q+1)|(S_i)_{H_i}|\ge |S_i|>(q+1)(2q-2)$, so $|(S_i)_{H_i}|\ge2q-1$.
The subsequence $(S_i)_{H_i}\mid S$ is regular over $G$, with $\supp((S_i)_{H_i})\subseteq H_i\setminus K$ and $H_i/K\cong C_q$.
By Lemma~\ref{lem:c0-short-block}, there exists $T_i\mid (S_i)_{H_i}$ with $|T_i|\le q$ and $\sigma(T_i)\in K\setminus\{0\}$.
Let $S_{i+1}=S_iT_i^{-1}$.

Next, let $R=S_{k+1}$. Assume to the contrary that $\Sigma_0(R)+K\ne G$, and let $H$ be a subgroup with $K<H<G$ and $|R_H|\ge q-1$.
Such a subgroup exists since $|R|\ge |S|-|S_K|-qk\ge2q^2-q-1>(q+1)(q-2)$.
Since $g+K\ne K$ for every $g\mid R_H$ and $|R_H|\ge q-1$, we have $\Sigma_0(R_H)+K=H$ by Lemma~\ref{lem2.2} in $H/K\cong C_q$.
We claim that $|RR_H^{-1}|\le q-2$. Assume to the contrary that $|RR_H^{-1}|\ge q-1$, and apply Lemma~\ref{lem2.2} in $G/H\cong C_q$, since $g+H\ne H$ for every $g\mid RR_H^{-1}$.
Then $\Sigma_0(RR_H^{-1})+H=G$ and hence $\Sigma_0(R)+K=\Sigma_0(RR_H^{-1})+H=G$, a contradiction.
Thus $|RR_H^{-1}|\le q-2$ and $|R_H|\ge2q^2-2q+1$.

Finally, we show by backward induction that $T_j\mid S_H$ for $j=k,\ldots,1$.
Suppose that $T_{j+1}\cdots T_k\mid S_H$. Since $S_j=T_j\cdots T_kR$, we have
$|S_j|-|(S_j)_H|\le |T_j|+|RR_H^{-1}|\le2q-2$, whereas $|(S_j)_H|\ge |R_H|\ge2q^2-2q+1$.
For every other subgroup $H'$ with $K<H'<G$, we have $H'\cap H=K$.
Since $(S_j)_K=1$, we have $|(S_j)_{H'}|\le2q-2<|(S_j)_H|$.
By the maximality of $|(S_j)_{H_j}|$, we have $H_j=H$ and hence $T_j\mid S_H$.
Since $SS_H^{-1}=RR_H^{-1}$ and $S$ is regular over $G$, we obtain $|H|=q|K|<|S|-q+2\le|S_H|\le|H|-1$, a contradiction.
\end{proof}

\begin{lem}\label{lem:c0-KneserW0}
Let $p$ be the smallest prime dividing $|G|$, let $K<G$ with $|K|\ge2$, and let $l\ge2$ be an integer.
Suppose that $S$ is regular over $G$ and admits a decomposition
$S=S_KT_1\cdots T_kR$, where $k=|K|-|S_K|$,
$|T_i|\le l$, $\sigma(T_i)=t_i\in K\setminus\{0\}$ and $\Sigma_0(R)+K=G$.
If
\[
|S|\ge\frac{|G|}{p}+l\frac{|K|}{p}-l+p-2,
\]
then $\Sigma(S)=G$.
\end{lem}

\begin{proof}
Let $W=S_Kt_1\cdots t_k$. Since $|W|=|K|\ge\mathsf D(K)$, we have $0\in\Sigma(W)$.
Since $t_i=\sigma(T_i)$ for $1\le i\le k$ and $S_KT_1\cdots T_k\mid S$, we have
$\Sigma(W)\subseteq\Sigma(S_KT_1\cdots T_k)\subseteq\Sigma(S)$. In particular, $0\in\Sigma(S)$.
If $W$ is regular over $K$, then $\Sigma_0(W)=K$ by Lemma~\ref{lem:c0-regular-sums}.
Since $S=S_KT_1\cdots T_kR$, we obtain $G=K+\Sigma_0(R)\subseteq\Sigma_0(S)=\Sigma(S)$.

So we assume that $W$ is not regular over $K$. Let $H_1<K$ be of minimum order with $|W_{H_1}|\ge |H_1|=m$.
Since $0\nmid W$, we have $|W_{\{0\}}|=0<1$, so $m>1$.
Every proper subgroup $H'<H_1$ has smaller order than $H_1$, so by the choice of $H_1$ we have
$|(W_{H_1})_{H'}|=|W_{H'}|\le |H'|-1$. Thus $W_{H_1}$ is regular over $H_1$.
Let $r=m-|S_{H_1}|-1\ge0$. At least $m-|S_{H_1}|=r+1$ of the $t_i$ belong to $H_1$,
so after relabeling the $T_i$ we have $t_1,\ldots,t_r\in H_1$.
The subsequence $S_{H_1}t_1\cdots t_r\mid W_{H_1}$ is regular over $H_1$ and has length $m-1$.
Thus $\Sigma_0(S_{H_1}t_1\cdots t_r)=H_1$ by Lemma~\ref{lem:c0-regular-sums}.
Let $W_0=St_1\cdots t_r(T_1\cdots T_r)^{-1}$ and $H_2=\st(\Sigma_0(W_0))$.
Since $S_{H_1}t_1\cdots t_r\mid W_0$, we have
$\Sigma_0(W_0)=H_1+\Sigma_0(W_0(S_{H_1}t_1\cdots t_r)^{-1})$. Hence $H_1\le H_2$.
Since $\{0,t_i\}\subseteq\Sigma_0(T_i)$ for $1\le i\le r$, we have $\Sigma_0(W_0)\subseteq\Sigma_0(S)=\Sigma(S)$.

If $H_2=G$, then $\Sigma_0(W_0)=G$ and $\Sigma(S)=G$. So we assume that $H_2<G$, whence $p\le |H_1|\le |H_2|\le |G|/p$.
Since $H_1<K$, we also have $r=|H_1|-|S_{H_1}|-1\le |H_1|-1\le |K|/p-1$.
Each of the $r$ replacements reduces the length by at most $l-1$, so $|W_0|\ge|S|-(l-1)r$.
By regularity, $|(W_0)_{H_2}|\le|S_{H_2}|+r\le|H_2|-1+r$.
By Lemma~\ref{lem2.1} and the lower bound on $|S|$,
\[
|\Sigma_0(W_0)|\ge(|S|-lr+2-|H_2|)|H_2|
\ge(|G|/p+p-|H_2|)|H_2|\ge|G|.
\]
The last inequality follows from $(|H_2|-p)(|G|/p-|H_2|)\ge0$, and thus $\Sigma(S)=G$.
\end{proof}

\begin{thm}\label{thm:unified-c0}
Let $G=C_q\oplus C_{nq}$, where $q$ is prime and $n\ge2$,
and let $p\ge3$ be the smallest prime dividing $|G|$. Then
\[
c_0(G)\le
\max\left\{nq+2q^2-q-1,\,
\frac{nq(q+1)}p-q+p-2\right\}.
\]
\end{thm}

\begin{proof}
Let $S$ be regular over $G$ with length at least the maximum in Theorem~\ref{thm:unified-c0}.
For $K=qG$, we have $|K|=n$ and $G/K\cong C_q^2$. Since $|S|\ge nq+2q^2-q-1$,
there is a decomposition $S=S_KT_1\cdots T_kR$ with $|T_i|\le q$ by Lemma~\ref{lem:c0-decomposition}.
Since $|S|\ge nq(q+1)/p-q+p-2$, we obtain $\Sigma(S)=G$ by applying Lemma~\ref{lem:c0-KneserW0} with $l=q$.
\end{proof}

\begin{cor}\label{thm5.1}
	Let $q$ be prime and let $G=C_q\oplus C_{nq^2}$, where $n\ge1$ is not a power of $q$. If $p\ge3$ is the smallest prime dividing $|G|$, then
	\[
		c_0(G)\le \max\bigl\{\frac{nq^2(q+1)}{p},nq^2 +f(q)-1\bigr\},
	\]
	where $f(q)=4q(q-1)+1$ if $q\ge7$ and $f(q)=2q^2+3$ if $q\le5$.
\end{cor}

\begin{proof}
We apply Theorem~\ref{thm:unified-c0} with $nq$ in place of $n$.
In both cases $f(q)\ge2q^2$, since $4q(q-1)\ge2q^2$ for $q\ge2$.
Thus $nq^2+2q^2-q-1\le nq^2+f(q)-1$. Also $nq^2(q+1)/p-q+p-2\le nq^2(q+1)/p$ since $p\le q$.
\end{proof}

\begin{cor}\label{thm1.2}
	Let $q$ be prime and let $G=C_q\oplus C_{nq}$, where $n\ge1$ and $q\nmid n$. If $p\ge3$ is the smallest prime dividing $|G|$, then
	\[
		c_0(G)\le \max\bigl\{(2q-1)(n+q+1),\frac{nq^2}{p} +2\frac{nq}{p}\bigr\}. 
	\]
\end{cor}

\begin{proof}
For $n=1$, we have $c_0(C_q^2)=2q-1\le(2q-1)(n+q+1)$ by \cite[Theorem~1.1(6)]{[GQZ2020]}.

So we assume that $n\ge2$. By Theorem~\ref{thm:unified-c0}, the bound for $n\ge2$ follows since
$(2q-1)(n+q+1)-(nq+2q^2-q-1)=(q-1)n+2q>0$
and, for $p\le q$, $nq(q+1)/p-q+p-2\le nq(q+1)/p\le nq^2/p+2nq/p$.
\end{proof}

\section{The zero-sum invariant $\mathsf N(G)$}

\begin{proof}[Proof of Theorem~\ref{thm1.4}]
To prove $\mathsf N(G)\ge r(n_r-1)+1$, let $(e_1,\ldots,e_r)$ be a basis of $G$ with $\ord(e_i)=n_i$ for $1\le i\le r$,
and define $g_0=e_r$ and $g_i=e_r+e_i$ for $1\le i<r$.
All $g_i$ have order $n_r$, so $S=\prod_{i=0}^{r-1}g_i^{n_r-1}$ is weak-regular.
For a zero-sum subsequence $T=\prod_{i=0}^{r-1}g_i^{a_i}$, comparing coordinates in $\sigma(T)=0$ shows that
$a_i\equiv0\pmod{n_i}$ for $i\ge1$ and $\sum_i a_i\equiv0\pmod{n_r}$.
Since $n_1\mid n_i$ for every $i$, these congruences imply $n_1\mid a_i$ for $0\le i<r$. Hence no $a_i$ equals $n_r-1$.
This proves $\mathsf{N}(G)\ge r(n_r-1)+1$.
Applying \eqref{eq:N-direct-sum} to $G=\bigoplus_{i=1}^r\langle e_i\rangle$, we obtain
$\mathsf N(G)\ge 1+\sum_{i=1}^r\bigl(\mathsf N(C_{n_i})-1\bigr)$.
\end{proof}

If $rn_r\ge2\sum_{i=1}^r n_i$, then $r(n_r-1)+1\ge 1+\sum_{i=1}^r\bigl(\mathsf N(C_{n_i})-1\bigr)$
by \eqref{eq:N-cyclic-bound} for $n_i\ge3$ and $\mathsf N(C_2)=2$.

In the proof of Corollary~\ref{cor:regular-main}, we apply the following criterion to $R=Sg^{-\mathsf v_g(S)}$, where $g\mid S$ has minimum multiplicity.

\begin{lem}\label{lem:occupied-directions}
Let $G=C_p\oplus C_{np}$, where $n\ge2$ and $p\ge3$ is the smallest
prime divisor of $np$. Let $2\le t\le p+1$ and let $R$ be regular over $G$.
Suppose that $\{g+pG:g\mid R,\ g\notin pG\}$ is contained in the union of $t$ distinct subgroups of order $p$ in $G/pG$.
If $|R|\ge np+t(p-2)+1$, then $\Sigma(R)=G$.
\end{lem}
\begin{proof}
We first construct disjoint subsequences $T_i\mid RR_{H_i}^{-1}$ of length $p-1$, where $H_1/pG,\ldots,H_t/pG$ are the $t$ subgroups in the hypothesis.
Let $R=g_1\cdots g_{|R|}$ and $A_i=\{j:g_j\notin H_i\}$. Since $R$ is regular and $|H_i|=np$,
we have $|A_i|=|R|-|R_{H_i}|\ge |R|-np+1\ge t(p-2)+2\ge p-1$.
For $i\ne k$, we have $H_i\cap H_k=pG$, so $A_i\cup A_k=\{j:g_j\notin pG\}$.
Since $pG\subseteq H_i$ for $1\le i\le t$, we have $\bigcup_{i\in I}A_i=\{j:g_j\notin\bigcap_{i\in I}H_i\}=\{j:g_j\notin pG\}$ for every $I\subseteq\{1,\ldots,t\}$ with $|I|\ge2$.
The union has $|R|-|R_{pG}|\ge |R|-n+1$ elements, since $|R_{pG}|\le |pG|-1=n-1$.
By the bound on $|R|$ and $t\le p+1$, we have $|R|-n+1-t(p-1)\ge n(p-1)+2-t\ge0$.
By Hall's theorem \eqref{eq:Hall}, there exist pairwise disjoint sets $J_i\subseteq A_i$ with $|J_i|=p-1$ for $1\le i\le t$. Let $T_i=\prod_{j\in J_i}g_j$.

Next, we decompose $R$ as $(U_1T_1)\cdots(U_tT_t)$ and show that $|\Sigma_0(U_iT_i)|\ge p(|U_i|+1)$ for each $1\le i\le t$.
Let $V=R(T_1\cdots T_t)^{-1}$, $U_1=V_{H_1}$, and $U_i=V_{H_i}(V_{pG})^{-1}$ for $2\le i\le t$.
Since $\supp(V)\subseteq\bigcup_{i=1}^t H_i$ and $H_i\cap H_j=pG$ for $i\ne j$, we have $V=U_1\cdots U_t$ with $U_i\mid R_{H_i}$ for each $i$. For $K<H_i$ we have
$|(U_i)_K|\le |R_K|\le |K|-1$, and $|U_i|\le |R_{H_i}|\le np-1$.
By Lemma~\ref{lem:c0-regular-sums} in $H_i$, $|\Sigma_0(U_i)|\ge |U_i|+1$.
Since $|T_i|=p-1$ and $g+H_i\ne H_i$ for every $g\mid T_i$, we have $\Sigma_0(T_i)+H_i=G$ by Lemma~\ref{lem2.2} in $G/H_i\cong C_p$.
Hence $\Sigma_0(U_iT_i)$ contains a translate of $\Sigma_0(U_i)$ in each
of the $p$ cosets, and $|\Sigma_0(U_iT_i)|\ge p(|U_i|+1)$.

Assume to the contrary that $\Sigma(R)\ne G$, so that $|\Sigma(R)|\le |G|-1$. By \eqref{eq:D-rank-two} with $n_1=p$ and
$n_2=np$, we have $\mathsf D(G)=np+p-1$. Since $|R|\ge np+2p-3\ge\mathsf D(G)$,
we have $0\in\Sigma(R)$ and $\Sigma(R)=\Sigma_0(R)$. By Lemma~\ref{lem5.2}, $\st(\Sigma(R))=\{0\}$, since $|R|\ge np+p-2$. By Lemma~\ref{lem2.1} applied to the sets $\Sigma_0(U_iT_i)$,
\[
 |\Sigma(R)|\ge p\sum_{i=1}^t(|U_i|+1)-t+1
 =p|R|-t(p-1)^2+1\ge np^2+p-t+1\ge |G|,
\]
a contradiction.
\end{proof}

Since $G/pG\cong C_p^2$ has $p+1$ subgroups of order $p$,
it follows from Lemma~\ref{lem:occupied-directions} with $t=p+1$ that
$c_0(C_p\oplus C_{np})\le np+p^2-p-1$ for the groups in Theorem~\ref{thm1.5}.
Qu and Han conjectured the value $np+2p-3$ for these groups%
~\cite[Concluding Remark]{[QH2017]}.

\begin{cor}\label{cor:regular-main}
Let $G=C_p\oplus C_{np}$, where $n\ge2$ and $p\ge3$ is the smallest prime
divisor of $np$. Let $S$ be a regular sequence over $G$. If $|S|\ge2np-1$,
or if $|S|=2np-2$ and $|\supp(S)|\ge3$, then $S$ has a nonempty zero-sum
subsequence $T$ such that $\mathsf v_g(T)=\mathsf v_g(S)$ for some $g\mid S$.
\end{cor}
\begin{proof}
Let $g\mid S$ have minimum multiplicity $a=\mathsf v_g(S)$, and let $R=Sg^{-a}$. We show that $\Sigma_0(R)=G$.
Since $\langle h\rangle<G$ for every $h\mid S$, regularity implies
$\mathsf v_h(S)\le\ord(h)-1\le np-1$. Hence $d=|\supp(S)|\ge3$ also
when $|S|\ge2np-1$. The subsequence $R$ is regular and $|R|\ge |S|(1-1/d)$.

For $n=p\ge5$, we have $c_0(G)=p^2+2p-3$ by Lemma~\ref{lem2.4}.
Since $d\ge3$ and $(p-1)(p-5)\ge0$, we obtain
$|R|\ge(4p^2-4)/3\ge p^2+2p-3$. Thus $\Sigma(R)=G$.

For $n=p=3$, we have $c_0(G)=12$ by Lemma~\ref{lem2.4}, so it suffices
to consider $|R|\le11$. Since $|R|\ge 2|S|/3$ and $|S|\ge16$,
we have $|S|=16$ and $|R|=11$. Assume to the contrary that $d\ge4$, so that $|R|\ge 3|S|/4=12>11=|R|$, a contradiction.
Hence $d=3$, $a=5$, and $R=x^6y^5$ after relabeling.
By regularity, $\ord(x)=\ord(y)=9$, so $3G=\langle 3x\rangle\le\langle x\rangle$.
Since $|R_{\langle x\rangle}|\le8<|R|$, we have $y\notin\langle x\rangle$.
Hence $y+3G\notin\langle x+3G\rangle$, and $G/3G=\langle x+3G\rangle\oplus\langle y+3G\rangle$.
Thus every coset of $3G$ contains an element $ix+jy$ with $0\le i,j\le2$.
Since $R=x^6y^5$ and $0\le i,j\le2$, we have $x^3y^3\mid R(x^iy^j)^{-1}$.
Since $3x$ and $3y$ are nonzero in $3G\cong C_3$,
$\{0,3x\}+\{0,3y\}=3G$. Hence $\Sigma_0(R)=G$.

So we assume that $n>p$. We apply Lemma~\ref{lem:occupied-directions} with
$t=\min\{d-1,p+1\}$, since $G/pG\cong C_p^2$ has $p+1$ subgroups of
order $p$ and $|\supp(R)|=d-1$. We have $n\ge p+2$,
since $n$ is odd. For $3\le d\le p+2$, the difference
\[
 (2np-2)(1-1/d)-np-(d-1)(p-2)-1
\]
is concave as a function of real $d$ and increases with $n$.
The difference equals $((p-2)^2+1)/3>0$ for $(n,d)=(p+2,3)$ and $p-1+2/(p+2)>0$ for $(n,d)=(p+2,p+2)$.
Hence the difference is positive for $n\ge p+2$ and $3\le d\le p+2$. For $d\ge p+2$, the length
bound $(2np-2)(1-1/d)$ is nondecreasing, whereas $t=p+1$ is fixed.
Thus $|R|\ge np+t(p-2)+1$, so $\Sigma(R)=G$ by Lemma~\ref{lem:occupied-directions}. These cases exhaust $n$, since $n\ge p$.

Since $-ag\in\Sigma_0(R)$, there exists $V\mid R$ such that $\sigma(V)=-ag$. Then $T=Vg^a$ is nonempty, $\sigma(T)=0$, and
$\mathsf v_g(T)=a=\mathsf v_g(S)$.
\end{proof}

\begin{lem}\label{lem:4.0}
Let $S$ be a non-regular sequence over a finite abelian group $G$.
Then there exists a proper subgroup $H<G$ such that $S_H$
is regular over $H$, $|S_H|\ge|H|$, and $\Sigma(S_H)=H$.
\end{lem}

\begin{proof}
Since $S$ is not regular over $G$, there exists a proper subgroup $H<G$ with $|S_H|\ge|H|$. Let $H$ have minimum order among these subgroups.
For every $H'<H$, we have $|H'|<|H|$ and hence $|(S_H)_{H'}|=|S_{H'}|\le|H'|-1$. Thus $S_H$ is regular over $H$.
If $H=\{0\}$, then $S_H$ contains $0$ and $\Sigma(S_H)=H$.
So we assume that $H\ne\{0\}$. Since $c_0(H)\le|H|\le|S_H|$ by Lemma~\ref{lem2.3}, we have $\Sigma(S_H)=H$.
\end{proof}

\begin{lem}\label{lem:N-heavy-fiber}
Let $S=S_HS_1$ with $\Sigma_0(S_H)=H$ and
$\supp(S_1)\subseteq G\setminus H$, and let $\varphi:G\to G/H$
be the natural homomorphism.
Suppose that some $\bar g\in\supp(\varphi(S_1))$ satisfies
$\mathsf v_{\bar g}(\varphi(S_1))\ge\ord(\bar g)$ and that
$|\varphi^{-1}(\bar g)\cap\supp(S_1)|\ge2$.
Then $S$ has a nonempty zero-sum subsequence $T$ with
$\mathsf v_g(T)=\mathsf v_g(S)$ for some $g\mid S_1$.
\end{lem}

\begin{proof}
We first construct $T_1\mid S_1$ with $\sigma(T_1)\in H$ and $\mathsf v_g(T_1)=\mathsf v_g(S_1)$ for some $g\mid S_1$.
Let $V$ be the subsequence of all terms $g\mid S_1$ with $\varphi(g)=\bar g$, and let $n=\ord(\bar g)$.
If $\mathsf v_h(V)\ge n$ for some $h\mid V$, let $|V|=kn+t$,
where $0\le t<n$. Then $T_1=Vh^{-t}$ satisfies $\varphi(\sigma(T_1))=kn\bar g=0$. Since $|\supp(V)|\ge2$, there is a $g\mid V$ with $g\ne h$, and $\mathsf v_g(T_1)=\mathsf v_g(S_1)$.
So we assume that $\mathsf v_h(V)<n$ for every $h\mid V$. For any $g\mid V$, we have $\mathsf v_g(V)<n\le|V|$. Thus there exists $T_1\mid V$ with $|T_1|=n$ and $\mathsf v_g(T_1)=\mathsf v_g(V)=\mathsf v_g(S_1)$. Then $\varphi(\sigma(T_1))=n\bar g=0$.
In both cases $T_1\mid S_1$ satisfies $\sigma(T_1)\in H$ and $\mathsf v_g(T_1)=\mathsf v_g(S_1)\ge1$ for some $g\mid S_1$.
Since $\Sigma_0(S_H)=H$ and $g\nmid S_H$, it follows from Lemma~\ref{lem2.9} that $S=S_HS_1$ has a nonempty zero-sum subsequence $T$ with $\mathsf v_g(T)=\mathsf v_g(S)$.
\end{proof}

\begin{lem}\label{lem:N-index-p-generators}
Let $G=C_p\oplus C_{np}$, where $p\ge3$ is the smallest prime
dividing $|G|$, $n\ge5$, and $H<G$ is cyclic of order $np$.
Let $S=S_Hg_0^b$ satisfy $g_0\notin H$, $\ord(g_0)=np$,
$b\le np-1$, $|S|\ge2np-2$, and $|S_H|\ge np$.
Suppose that $\langle g\rangle=H$ for every $g\mid S_H$.
Let $h\mid S_H$ satisfy $a=\mathsf v_h(S_H)=\min_{x\mid S_H}\mathsf v_x(S_H)$.
Let $R=S_Hh^{-a}$. If $|R|\le np-1$ and $a<np/2$,
then $S$ has a nonempty zero-sum subsequence $T$ with $\mathsf v_h(T)=\mathsf v_h(S)$.
\end{lem}

\begin{proof}
To show that $H\subseteq\Sigma_0(Rg_0^{p\lfloor b/p\rfloor})$, let
$g_1,\ldots,g_l$ be the distinct terms of $R$, with $b=ps_g+t_g$ and
$\mathsf v_{g_i}(R)=ps_i+t_i$, where $0\le t_i,t_g<p$.
First suppose that $a\ge2p$.
Since $la\le|R|\le np-1$ and $|R|+b=|S|-a$,
$p(s_g+\sum_i s_i)\ge2np-2-a-(l+1)(p-1)$.
The function $f(x)=x+(p-1)(np-1)/x$ decreases on $(0,\sqrt{(p-1)(np-1)}]$,
increases on $[\sqrt{(p-1)(np-1)},\infty)$, and  $f(x) \le np/2+2p$ at both endpoints of $[2p,np/2]$. Hence
\[
p\Big(s_g+\sum_i s_i\Big)
\ge 3np/2-3p-1>(n-1)p,
\]
where the strict inequality follows from $n\ge5$ and $p\ge3$.
Thus $s_g+\sum_i s_i\ge n$.
Since $s_g\le n-1$, there exists $i$ with $s_i>0$.
Let $U=(pg_0)^{s_g}\prod_{j=1}^l(pg_j)^{s_j}$. For $0\le j\le l$, we have $\ord(pg_j)=\ord(g_j)/p=n$.
Since $pH\le pG$ and $|pH|=|pG|=n$, each term of $U$ generates $pG=pH$.
Thus $U(pg_i)^{-1}$ is regular over $pH$ and has length at least $n-1$, so
$\Sigma_0(U(pg_i)^{-1})=pH$ by Lemma~\ref{lem:c0-regular-sums}.
Since $\sigma(g_j^p)=pg_j$ for $0\le j\le l$ and $s_i>0$, we have
$\Sigma_0(U(pg_i)^{-1})\subseteq\Sigma_0(R(g_i^p)^{-1}g_0^{ps_g})$.
Hence $H=\Sigma_0(g_i^p)+pH\subseteq\Sigma_0(Rg_0^{ps_g})$.

So we assume that $a\le2p-1$.
Since $b,|R|\le np-1$ and $|R|+b\ge2np-2-a$, we have
$\min\{b,|R|\}\ge np-1-a\ge(n-2)p$. Hence $s_g\ge n-2$ and $|R|\ge3p$.
In particular, there exist $U,V$ such that $UV\mid R$ and $|U|=|V|=p$.
Since $\langle z\rangle=H$ for every $z\mid U$, the sequence $U$ is regular over $H$. By Lemma~\ref{lem:c0-regular-sums}, $|\Sigma_0(U)|\ge p+1$. Thus some distinct $u_1,u_2\in\Sigma_0(U)$ satisfy $u_1-u_2\in pH$.
Since $z+pH\ne pH$ for every $z\mid V$ and $|V|=p$, we have $\Sigma_0(V)+pH=H$ by Lemma~\ref{lem2.2} in $H/pH\cong C_p$. Hence there exists $v\in\Sigma_0(V)$ with $u_1+v\in pH$. Since $u_1+v$ and $u_2+v$ are distinct elements of $pH$, at least one is nonzero.
Thus there is $W\mid UV$ with $\sigma(W)\in pH\setminus\{0\}$.
Since $pg_0$ generates $pH$ and $\sigma(W)\ne0$, each proper
subgroup of $pH$ contains at most one term of $(pg_0)^{s_g}\sigma(W)$,
and the trivial subgroup contains none. The sequence is regular over
$pH$ and has length at least $n-1$. By Lemma~\ref{lem:c0-regular-sums} and $\sigma(g_0^p)=pg_0$,
we have $pH=\Sigma_0((pg_0)^{s_g}\sigma(W))\subseteq\Sigma_0(Wg_0^{ps_g})$.
Since $|RW^{-1}|\ge p$ and $z+pH\ne pH$ for every $z\mid RW^{-1}$,
we have $\Sigma_0(RW^{-1})+pH=H$ by Lemma~\ref{lem2.2} in $H/pH$.
The subsequences $Wg_0^{ps_g}$ and $RW^{-1}$ are disjoint in $Rg_0^{ps_g}$, so
$H=pH+\Sigma_0(RW^{-1})\subseteq\Sigma_0(Rg_0^{ps_g})$.

In both cases, $-ah\in\Sigma_0(Rg_0^{ps_g})$, so there exists $T_1\mid Rg_0^{ps_g}$ with $\sigma(T_1)=-ah$. Then $T=T_1h^a$ is zero-sum and $\mathsf v_h(T)=\mathsf v_h(S)$.
\end{proof}

\begin{lem}\label{lem:N-index-p-low-order}
Let $G=C_p\oplus C_{np}$, where $n\ge2$ and $p\ge3$ is the
smallest prime dividing $|G|$, and let $H<G$ be cyclic of order $np$.
Let $S=S_Hg_0^b$ satisfy $g_0\notin H$, $\ord(g_0)=np$,
$b\le np-1$, $|S|\ge2np-2$, and suppose that $S_H$ is regular over $H$.
If $\ord(h')\le n$ for some $h'\mid S_H$, then $S$ has a nonempty zero-sum
subsequence $T$ with $\mathsf v_{h'}(T)=\mathsf v_{h'}(S)$.
\end{lem}

\begin{proof}
We prove that $\Sigma_0(U)=H$ for $U=R'(pg_0)^s$, where
$a=\mathsf v_{h'}(S)$, $R'=S_H(h')^{-a}$, and $s=\lfloor b/p\rfloor$.
By regularity, $a\le |S_{\langle h'\rangle}|\le\ord(h')-1\le n-1$, while $s\le n-1$ and
$b-s\le n(p-1)$. Hence $|U|=|S|-a-b+s\ge np-1$.
Assume to the contrary that $\Sigma_0(U)\ne H$, and let $H_1=\st(\Sigma_0(U))$ and $d=|H_1|$.
Assume to the contrary that $H_1=\{0\}$, and let $U=u_1\cdots u_l$. Since $0\nmid U$ and $\Sigma_0(U)\ne H$, Lemma~\ref{lem2.1} applied to $\Sigma_0(U)=\sum_{j=1}^l\{0,u_j\}$ yields
$|H|-1\ge|\Sigma_0(U)|\ge |U|+1\ge np=|H|$, a contradiction. Thus $1<d<np$.
Since $H\setminus\Sigma_0(U)$ contains a coset of $H_1$, $|\Sigma_0(U)|\le np-d$. By Lemma~\ref{lem2.1},
$(|U|-|U_{H_1}|+1)d\le |\Sigma_0(U)|$. Since $R'$ is regular over $H$, we have $|R'_{H_1}|\le d-1$. Hence
\[
|U|-|U_{H_1}|\le np/d-2,\qquad
|U_{H_1}|\le d-1+s\le d+n-2.
\]
Since $p\le d\le n$,
$|U|\le np/d+d+n-4\le3n-4<np-1$, a contradiction.
Thus $\Sigma_0(U)=H\subseteq\Sigma_0(R'g_0^{ps})$.
Hence there exists $T_1\mid R'g_0^{ps}$ with $\sigma(T_1)=-ah'$. Then $T=T_1(h')^a$ is zero-sum and $\mathsf v_{h'}(T)=\mathsf v_{h'}(S)$.
\end{proof}

\begin{lem}\label{lem:N-half-quotient}
Let $H=C_m$, let $R$ be regular over $H$ with $|R|\ge m/2$,
and let $H_0\le H$ be nontrivial. Then $\Sigma_0(R)+H_0=H$.
\end{lem}
\begin{proof}
Assume to the contrary that $\Sigma_0(R)+H_0\ne H$, and let $H_1=\st(\Sigma_0(R)+H_0)$ in $H$.
Then $H_0\le H_1<H$, so $k=|H_1|$ satisfies $2\le k\le m/2$. Hence $(k-2)(m/2-k)\ge0$, or equivalently, $k+m/k\le 2+m/2$.
Since $H\setminus(\Sigma_0(R)+H_0)$ contains a coset of $H_1$, we have $|\Sigma_0(R)+H_0|\le m-k$.
Let $R=g_1\cdots g_l$. By Lemma~\ref{lem2.1},
$|\Sigma_0(R)+H_0|=\bigl|H_0+\sum_{i=1}^l\{0,g_i\}\bigr|\ge (|R|-|R_{H_1}|+1)k$. Since $|R_{H_1}|\le k-1$, we obtain
\[
 |R|-|R_{H_1}|\le m/k-2,\qquad
 m/2\le |R|\le k+m/k-3\le m/2-1,
\]
a contradiction.
\end{proof}

\begin{lem}\label{lem:N-kernel-blocks}
Let $S=S_HS_1$ be a sequence over an abelian group $G$, where
$H=C_m$, $m\ge2$, $S_H$ is regular over $H$, $|S_H|\ge m$,
and $\mathsf v_h(S_H)<\ord(h)$ for every $h\mid S_H$.
Let $T_1,\ldots,T_l$ be nonempty sequences such that $T_1\cdots T_l\mid S_1$
and $\sigma(T_i)\in H$ for every $i$. Let $U=\sigma(T_1)\cdots\sigma(T_l)$.
If $\Sigma_0(U)$ contains a nontrivial subgroup $H_0\le H$, then $S$
has a nonempty zero-sum subsequence $T$ with
$\mathsf v_h(T)=\mathsf v_h(S)$ for some $h\mid S_H$.
In particular, this holds if $0\nmid U$, $\supp(U)\subseteq K$ for a nontrivial subgroup $K\le H$, and $|U|\ge |K|-1$.
\end{lem}
\begin{proof}
Since $\mathsf v_h(S_H)<\ord(h)\le m\le|S_H|$ for every $h\mid S_H$, we have $|\supp(S_H)|\ge2$.
Let $h\mid S_H$ satisfy $a=\mathsf v_h(S_H)=\min_{x\mid S_H}\mathsf v_x(S_H)$. Then $R=S_Hh^{-a}$ is regular over $H$ and $|R|\ge|S_H|/2\ge m/2$.
By Lemma~\ref{lem:N-half-quotient}, $\Sigma_0(R)+H_0=H$, so there exists $V\mid R$ with $\sigma(V)+ah\in H_0$.
Since $H_0\subseteq\Sigma_0(U)$, there is a set $I\subseteq\{1,\ldots,l\}$
such that $\sum_{i\in I}\sigma(T_i)=-\sigma(V)-ah$.
Thus $T=Vh^a\prod_{i\in I}T_i$ satisfies $\sigma(T)=0$, and $T\mid S$ because $Vh^a\mid S_H$ and $\prod_{i\in I}T_i\mid S_1$. Since $h\nmid R$ and $h\nmid S_1$, we have $\mathsf v_h(T)=a=\mathsf v_h(S)$.
For the last assertion, $\st(\Sigma_0(U))\subseteq\Sigma_0(U)$ since $0\in\Sigma_0(U)$.
If $\st(\Sigma_0(U))=\{0\}$, then $|\Sigma_0(U)|\ge |U|+1\ge |K|$ by Lemma~\ref{lem2.1}, since $0\nmid U$. Thus $\Sigma_0(U)=K$.
\end{proof}

\begin{lem}\label{lem:N-cyclic-quotient}
Let $G=C_p\oplus C_{np}$, where $n\ge2$ and $p\ge3$ is the
smallest prime dividing $|G|$.
Let $S=S_HS_1$ satisfy $|S|\ge2np-2$ and
$\mathsf v_g(S)<\ord(g)$ for every $g\mid S$.
Suppose that $H<G$ is cyclic of order $m$, that $S_H$ is regular over $H$,
and that $m\le|S_H|\le\mathsf N(H)-1$.
If $G/H$ is cyclic of order greater than $p$, then $S$ has a nonempty
zero-sum subsequence $T$ with $\mathsf v_g(T)=\mathsf v_g(S)$
for some $g\mid S$.
\end{lem}
\begin{proof}
The equality $\Sigma_0(S_H)=H$ follows from Lemma~\ref{lem:c0-regular-sums}, since $S_H$ is regular over $H$ and $|S_H|\ge m$.
Since $H$ is cyclic of order $m$, we have $m\mid\exp(G)=np$. Let $k=np/m$, so that $|G/H|=pk>p$, and let $\varphi:G\to G/H$
be the natural homomorphism.
Since $G/pG\cong C_p^2$ and $(G/H)/p(G/H)\cong C_p$,
the kernel $(H+pG)/pG\cong H/(H\cap pG)$ of the natural homomorphism
$G/pG\to(G/H)/p(G/H)$ has order $p$. Thus $p\mid m$ and,
as $H$ is cyclic, $H\cap pG=pH$.
Since $1<k\mid np$, we have $k\ge p$. Also $|S_H|\le\mathsf N(C_m)-1$, so $|S_1|=|S|-|S_H|\ge2np-1-\mathsf N(C_m)$.

We next construct $R\mid S_1$ such that $\varphi(R)$ is weak-regular. By
Lemma~\ref{lem:N-heavy-fiber}, we may assume that $|\varphi^{-1}(\bar g)\cap\supp(S_1)|=1$ whenever
$\mathsf v_{\bar g}(\varphi(S_1))\ge\ord(\bar g)$.
For each $g\mid S_1$, let $r_g=\ord(\varphi(g))$. Let
$T_0=\prod_{g\in\supp(S_1)}(g^{r_g})^{\lfloor\mathsf v_g(S_1)/r_g\rfloor}$.
Since $r_g\lfloor\mathsf v_g(S_1)/r_g\rfloor\le\mathsf v_g(S_1)$ for every $g\mid S_1$, we have $T_0\mid S_1$.
Let $R=S_1T_0^{-1}$ and $U=\prod_{g\in\supp(S_1)}(r_g g)^{\lfloor\mathsf v_g(S_1)/r_g\rfloor}$.
For every $g\mid T_0$, we have $r_g g\in H$ and $r_g\le\mathsf v_g(S_1)<\ord(g)$, so $0\nmid U$.
If $\mathsf v_{\bar g}(\varphi(S_1))\ge\ord(\bar g)$, let $g$ be the unique element of $\varphi^{-1}(\bar g)\cap\supp(S_1)$.
Then $\mathsf v_{\bar g}(\varphi(R))=\mathsf v_g(S_1)-r_g\lfloor\mathsf v_g(S_1)/r_g\rfloor<r_g=\ord(\bar g)$.
For every other $\bar g\mid\varphi(S_1)$, we have $\mathsf v_{\bar g}(\varphi(R))\le\mathsf v_{\bar g}(\varphi(S_1))<\ord(\bar g)$. Thus $\varphi(R)$ is weak-regular.

If $|R|\ge\mathsf N(G/H)$, then by the definition of $\mathsf N(G/H)$ there exists $R_1\mid R$ such that $\sigma(\varphi(R_1))=0$ and
$\mathsf v_{\bar g}(\varphi(R_1))=\mathsf v_{\bar g}(\varphi(R))$
for some $\bar g\mid\varphi(R)$.
Then $\mathsf v_h(R_1)=\mathsf v_h(R)$ for every $h\mid R$ with $\varphi(h)=\bar g$.
Since $\sigma(R_1),\sigma(T_0)\in H$ and $S_1=RT_0$, the subsequence $T_1=R_1T_0$ satisfies $\sigma(T_1)\in H$ and
$\mathsf v_h(T_1)=\mathsf v_h(S_1)$ for every
$h\in\varphi^{-1}(\bar g)\cap\supp(S_1)$.
Since $\Sigma_0(S_H)=H$, Lemma~\ref{lem2.9} yields a nonempty zero-sum subsequence $T\mid S$ with $\mathsf v_h(T)=\mathsf v_h(S)$ for some $h\mid S_1$.

So we assume that $|R|<\mathsf N(G/H)=\mathsf N(C_{pk})$, and distinguish two cases according to the order $m$ of $H$.

\smallskip\noindent\emph{Case 1: $m>p$.}
Since $1<m/p\mid |G|$, we have $m\ge p^2$. By \eqref{eq:N-cyclic-bound},
$\mathsf N(C_m)\le2m-1$, so $|S_1|\ge2np-2m$.
By Lemma~\ref{lem:N-kernel-blocks}, we may assume that $\Sigma_0(U)$
contains no nontrivial subgroup of $H$.
Let $U_1=U_{pH}$, $U_2=UU_1^{-1}$, $a=|U_1|$, and $b=|U_2|$.
By Lemma~\ref{lem:N-kernel-blocks} applied to $U_1$ with $K=pH$, we may assume $a\le |pH|-2=m/p-2$.
For every $g\mid T_0$, we have $r_g\mid |G/H|=pk$ and hence $r_g\le pk$.
If $r_g g\notin pH$, then $p\nmid r_g$, since otherwise $r_g g\in H\cap pG=pH$. Thus $r_g\mid k$ and $r_g\le k$.
If $U_2\ne1$, there exists $g\mid T_0$ with $r_g g\notin pH$, and $r_g>1$ since $g\notin H$.
For any prime $q\mid r_g$, we have $q\mid k$ and $q>p\ge3$, so $p+2\le q\le k$.

If $b\le p-2$, then $|T_0|\le k(pa+b)\le np-(p+2)k$ and $ |R|\ge m(k-2)+(p+2)k\ge2pk,$ 
where $m(k-2)\ge p(k-2)=k(p-2)+2(k-p)\ge k(p-2)$.
Thus $|R|\ge2pk\ge\mathsf N(C_{pk})$, a contradiction.

So we assume that $b=|U_2|\ge p-1$. Then there exists $V\mid U_2$ with $|V|=p-1$.
$h+pH\ne pH$ for every $h\mid V$, so $\Sigma_0(V)+pH=H$ by Lemma~\ref{lem2.2} in $H/pH\cong C_p$.
Since $0\in\Sigma_0(U)$, we have $\st(\Sigma_0(U))\subseteq\Sigma_0(U)$ and hence $\st(\Sigma_0(U))=\{0\}$.
For every $W\mid U$, we have $\st(\Sigma_0(W))\le\st(\Sigma_0(U))=\{0\}$ since
$\Sigma_0(U)=\Sigma_0(W)+\Sigma_0(UW^{-1})$.
Since $0\nmid U_1$, Lemma~\ref{lem2.1} yields $|\Sigma_0(U_1)|\ge a+1$.
Since $\Sigma_0(U_1)\subseteq pH$ and $\Sigma_0(V)+pH=H$, we have
$|\Sigma_0(U_1V)|\ge p|\Sigma_0(U_1)|\ge p(a+1)$.
For $U_1V\mid W\mid U$ and $u\mid UW^{-1}$, we have $|\Sigma_0(Wu)|\ge|\Sigma_0(W)|+1$.
Assume to the contrary that $\Sigma_0(Wu)=\Sigma_0(W)$, so that $u+\Sigma_0(W)=\Sigma_0(W)$ and $u\in\st(\Sigma_0(W))=\{0\}$ although $u\ne0$, a contradiction.
It follows that
\[
 pa+b+1\le|\Sigma_0(U)|\le m-1,\qquad |T_0|\le k(pa+b)\le np-2k.
\]
Since $m\ge p^2$, $k\ge p+2$, and $p^2(k-2)-2k(p-1)\ge(p-2)(p^2-2)>0$, we have
$|R|\ge m(k-2)+2k\ge2pk\ge\mathsf N(C_{pk})$, a contradiction.

\smallskip\noindent\emph{Case 2: $m=p$.}
For every $g\mid T_0$, we have $r_g g\ne0$. Assume to the contrary that $p\mid r_g$ for some $g\mid T_0$, so that $r_g g\in H\cap pG=pH=\{0\}$, a contradiction.
Hence $p\nmid r_g$ for every $g\mid T_0$. Since $r_g\mid pk$, we obtain $r_g\mid k$ and thus $r_g\le k$.
If $|U|\ge p-2$, then $|\Sigma_0(U)|\ge p-1$ by Lemma~\ref{lem2.2} in $H=C_p$.
Since $|S_H|\ge p$ and $\mathsf v_g(S_H)<p$ for every $g\mid S_H$, there exist distinct $h,x\in\supp(S_H)$.
By Lemma~\ref{lem2.2} applied to $\Sigma_0(U)$ and $\{0,x\}$, we have $\Sigma_0(U)+\{0,x\}=H$.
Since $\Sigma_0(U)\subseteq\Sigma_0(T_0)$, there is $T_1\mid T_0x$
with $\sigma(T_1)=-\mathsf v_h(S)h$.
Then $T=T_1h^{\mathsf v_h(S)}$ is zero-sum and
$\mathsf v_h(T)=\mathsf v_h(S)$.
So we assume that $|U|\le p-3$.

If $T_0=1$, then $R=S_1$. Since $pk\ge p^2$, we have $\lceil2\sqrt{pk+1}\rceil\ge2p+1$. By \eqref{eq:N-cyclic-bound},
\[
 |R|\ge2pk-1-\mathsf N(C_p)\ge2pk-2p\ge\mathsf N(C_{pk}),
\]
a contradiction.
So we assume that $T_0\ne1$. Then there exists $g\mid S_1$ with $\mathsf v_g(S_1)\ge r_g$.
Let $t=\mathsf v_g(R)<r_g\le k$ and $R'=Rg^{-t}$.
We have $\varphi^{-1}(\varphi(g))\cap\supp(S_1)=\{g\}$. Thus $\varphi(g)\nmid\varphi(R')$ and the subsequence $\varphi(R')\mid\varphi(R)$ is weak-regular. Since $\mathsf N(C_p)\le2p-1\le2k-1$, we have
\[
 |R'|\ge2pk-1-\mathsf N(C_p)-(p-3)k-(k-1)
 =(p+2)k-\mathsf N(C_p)\ge pk.
\]
If $\Sigma(\varphi(R'))=G/H$, there exists $R_1\mid R'$ such that
$\varphi(\sigma(R_1))=-\mathsf v_g(S_1)\varphi(g)$.
Then $T_1=R_1g^{\mathsf v_g(S_1)}$ satisfies $\sigma(T_1)\in H$ and
$\mathsf v_g(T_1)=\mathsf v_g(S_1)$. By Lemma~\ref{lem2.9}, there is a nonempty zero-sum subsequence $T\mid S$ with $\mathsf v_g(T)=\mathsf v_g(S)$.
So we assume that $\Sigma(\varphi(R'))\ne G/H$. By Lemma~\ref{lem2.8} applied to $\varphi(R')$, there exists $R_1\mid R'$ with $\sigma(\varphi(R_1))=0$ and
$\mathsf v_{\bar h}(\varphi(R_1))=\mathsf v_{\bar h}(\varphi(R'))$
for some $\bar h\mid\varphi(R')$.
Since $\bar h\ne\varphi(g)$, we have $\mathsf v_{\bar h}(\varphi(R'))=\mathsf v_{\bar h}(\varphi(R))$.
Hence $\sigma(R_1T_0)\in H$ and $\mathsf v_h(R_1T_0)=\mathsf v_h(S_1)$ for every $h\mid S_1$ with $\varphi(h)=\bar h$.
Since $\Sigma_0(S_H)=H$, it follows from Lemma~\ref{lem2.9} with $T_1=R_1T_0$ that $S$ has a nonempty zero-sum subsequence $T$ satisfying $\mathsf v_h(T)=\mathsf v_h(S)$ for some $h\mid S_1$ with $\varphi(h)=\bar h$.
\end{proof}

\begin{lem}\label{lem:N-index-p-complete}
Let $G=C_p\oplus C_{np}$, where $n\ge2$ and $p\ge3$ is the
smallest prime dividing $|G|$.
Let $S$ satisfy $|S|\ge2np-2$ and
$\mathsf v_g(S)<\ord(g)$ for every $g\in\supp(S)$.
Suppose that $H<G$ is cyclic of order $np$, that $S_H$ is regular
over $H$, and that $np\le|S_H|\le\mathsf N(H)-1$.
Then $S$ has a nonempty zero-sum subsequence $T$ such that
$\mathsf v_g(T)=\mathsf v_g(S)$ for some $g\in\supp(S)$.
\end{lem}

\begin{proof}
Let $S_1=SS_H^{-1}$ and let $\varphi:G\to G/H\cong C_p$ be the natural homomorphism.
We first consider the case in which $\varphi(S_1)$ is weak-regular.
By Lemma~\ref{lem:c0-regular-sums}, $\Sigma_0(S_H)=H$.
By Lemma~\ref{lem2.9}, it suffices to find $T_1\mid S_1$ such that $\sigma(T_1)\in H$ and
$\mathsf v_g(T_1)=\mathsf v_g(S_1)$ for some $g\mid S_1$. Since $n\ge p$, it follows from \eqref{eq:N-cyclic-bound} that
\[
 |S_1|\ge2np-1-\mathsf N(C_{np})
 \ge\lceil2\sqrt{np+1}\rceil-2\ge2p-1\ge\mathsf N(C_p).
\]
If $\varphi(S_1)$ is weak-regular, then by the definition of $\mathsf N(C_p)$ there exists $T_1\mid S_1$ such that $\sigma(\varphi(T_1))=0$ and
$\mathsf v_{\bar g}(\varphi(T_1))=\mathsf v_{\bar g}(\varphi(S_1))$
for some $\bar g\mid\varphi(S_1)$.
Thus $\sigma(T_1)\in H$ and $\mathsf v_g(T_1)=\mathsf v_g(S_1)$ for every $g\mid S_1$ with $\varphi(g)=\bar g$.

So we assume that $\varphi(S_1)$ is not weak-regular. Since $0\nmid\varphi(S_1)$ and $G/H\cong C_p$, there exists $\bar g\mid\varphi(S_1)$ with $\mathsf v_{\bar g}(\varphi(S_1))>p$.
If $|\varphi^{-1}(\bar g)\cap\supp(S_1)|\ge2$, Lemma~\ref{lem:N-heavy-fiber} applies.
So we assume that $\varphi^{-1}(\bar g)\cap\supp(S_1)=\{g\}$ for some $g\mid S_1$.

If $|\supp(S_1)|\ge2$, let $g_0\in\supp(S_1)\setminus\{g\}$.
Since $\ord(g+H)=p$ and $\Sigma_0(S_H)=H$, we have
$\Sigma_0(g^pS_H)=\bigcup_{j=0}^{p-1}(jg+H)=G$.
Hence there exists $T'\mid g^pS_H$ with $\sigma(T')=-\mathsf v_{g_0}(S_1)g_0$.
Since $g_0\nmid g^pS_H$, the sequence
$T=T'g_0^{\mathsf v_{g_0}(S_1)}$ is a zero-sum subsequence of $S$
with $\mathsf v_{g_0}(T)=\mathsf v_{g_0}(S)$.

So we assume that $S_1=g_0^b$, where $b\ge p$ and
$b<\ord(g_0)\le np$.
Let $h\mid S_H$ satisfy $a=\mathsf v_h(S_H)=\min_{x\mid S_H}\mathsf v_x(S_H)$, and let $R=S_Hh^{-a}$.
Since $\mathsf v_x(S_H)=\mathsf v_x(S)<\ord(x)\le np\le |S_H|$ for every $x\mid S_H$, we have $|\supp(S_H)|\ge2$.
By the minimality of $a$, $|S_H|\ge a|\supp(S_H)|\ge2a$, so $|R|=|S_H|-a\ge |S_H|/2$.
If $|R|\ge np$, then $\Sigma_0(R)=H$ by Lemma~\ref{lem:c0-regular-sums}.
Hence there exists $T_1\mid R$ with $\sigma(T_1)=-ah$. Then $T=T_1h^a$
has $\mathsf v_h(T)=a=\mathsf v_h(S)$.
So we assume that $|R|\le np-1$.

\smallskip\noindent\emph{Case 1: $|\supp(R)|\ge2$.}
Then $2a\le|R|\le np-1$ and
$b=|S|-|R|-a\ge np-1-a\ge(np-1)/2\ge n$.
Since $p$ is the smallest prime divisor of $np$, every proper divisor of $np$ is at most $n$. Thus $\ord(g_0)\mid np$ and $\ord(g_0)>b\ge n$ imply $\ord(g_0)=np$.
If $\langle z\rangle\ne H$ for some $z\mid S_H$, then $\ord(z)\le np/p=n$. By Lemma~\ref{lem:N-index-p-low-order} with $h'=z$, there is a nonempty zero-sum subsequence $T\mid S$ with $\mathsf v_z(T)=\mathsf v_z(S)$.
So we assume that $\langle z\rangle=H$ for every $z\mid S_H$.
By Lemma~\ref{lem:N-index-p-generators}, the conclusion holds for $n\ge5$, so it remains to consider $p=n=3$.
We show that $Rg_0^b$ is regular over $G$. If $K<G$ and $K\cap\supp(R)\ne\emptyset$, then $H\le K$. Since $[G:H]=3$, we have $K=H$ and $|(Rg_0^b)_K|=|R|\le8=|K|-1$.
For every proper subgroup $K\ne H$ with $K\cap\supp(Rg_0^b)\ne\emptyset$, we have $(Rg_0^b)_K=g_0^b$ and $b<\ord(g_0)\le|K|$.
Since $a\le4$, we have $|Rg_0^b|\ge12=c_0(C_3\oplus C_9)$ by Lemma~\ref{lem2.4}.
Hence there is $T_1\mid Rg_0^b$ with $\sigma(T_1)=-ah$,
and $T=T_1h^a$ has the required property.

\smallskip\noindent\emph{Case 2: $|\supp(R)|=1$.}
Then $R=f^c$ for some $f\in H$, with $a\le c\le np-1$.
Since $2c\ge|S_H|\ge np$ and $c<\ord(f)$, we have $\ord(f)=np$.
Let $c=ps_f+t_f$ and $b=ps_g+t_g$, where $0\le t_f,t_g<p$.
If $a\ge(n-1)p$, then $s_f\ge n-1$ and $s_g\ge1$.
So we assume that $a\le(n-1)p-1$. Then $c+b\ge(n+1)p-1$ and
\[
 s_f+s_g\ge
 \left\lceil\frac{(n+1)p-1-2(p-1)}p\right\rceil=n.
\]
Thus $s_f+s_g\ge n$ in both cases, and $s_f\ge\lfloor n/2\rfloor\ge1$.
The sequence $V=(pf)^{s_f-1}(pg_0)^{s_g}$ over $pH$ has length $|V|=s_f+s_g-1\ge n-1$.
Since $\langle f\rangle=H$, we have $\langle pf\rangle=pH=pG\cong C_n$.
Since $g_0+H$ has order $p$, we have $p\mid\ord(g_0)$, and the integer
$d=\ord(pg_0)=\ord(g_0)/p$ satisfies $d>1$ and $s_g\le d-1$.
For every proper subgroup $K<pH$, we have $pf\notin K$. If $pg_0\in K$, then $|V_K|=s_g\le d-1<|K|$. Otherwise $|V_K|=0$.
Thus $V$ is regular over $pH$, and $\Sigma_0(V)=pH$ by Lemma~\ref{lem:c0-regular-sums}.
Since $\sigma(f^p)=pf$ and $\sigma(g_0^p)=pg_0$, we have $\Sigma_0(V)\subseteq\Sigma_0(f^{p(s_f-1)}g_0^{ps_g})$.
Since $\Sigma_0(f^p)+pH=H$, we obtain
\[
 H=\Sigma_0(V)+\Sigma_0(f^p)
 \subseteq\Sigma_0(f^{ps_f}g_0^{ps_g})
 \subseteq\Sigma_0(f^cg_0^b).
\]
Hence there exists $T_1\mid f^cg_0^b$ with $\sigma(T_1)=-ah$.
Then $T=T_1h^a\mid S$ satisfies $\sigma(T)=\sigma(T_1)+ah=0$ with $\mathsf v_h(T)=\mathsf v_h(S)=a\ge1$.
\end{proof}

\begin{proof}[Proof of Theorem~\ref{thm1.5}]
By Theorem~\ref{thm1.4}, $\mathsf N(G)\ge 2np-1$. We prove $\mathsf N(G)\le 2np-1$ by induction on $n$.
For $k\mid n$ with $2\le k<n$, we may assume that $\mathsf N(C_p\oplus C_{kp})=2kp-1$, and for $k=1$ we use \eqref{eq:N-elementary-rank-two} for $C_p^2$. Let $S$ be weak-regular over $G$ with $|S|\ge2np-2$.
If $\mathsf v_g(S)=\ord(g)$ for some $g\mid S$, then
$T=g^{\ord(g)}$ is zero-sum and $\mathsf v_g(T)=\mathsf v_g(S)$.
So we assume that $\mathsf v_g(S)<\ord(g)$ for all $g\mid S$. In particular, $0\nmid S$.

If $|\supp(S)|\le2$, then $2np-2\le |S|\le 2(\exp(G)-1)=2np-2$,
so $S=g^{np-1}h^{np-1}$ for some $g\ne h$ of order $np$.
If $\langle g\rangle=\langle h\rangle$, then $g=jh$ for some $1\le j<np$, so $T=g^{np-1}h^j$ satisfies $\sigma(T)=0$ and $\mathsf v_g(T)=np-1=\mathsf v_g(S)$.
So we assume that $\langle g\rangle\ne\langle h\rangle$.
Assume to the contrary that a zero-sum subsequence $T\mid S$ satisfies $\mathsf v_g(T)=np-1$, so that $(np-1)g+jh=0$ and hence $g=jh$ for some $0\le j<np$.
Since $g=jh$, we have $\langle g\rangle\le\langle h\rangle$. Both subgroups have order $np$, so $\langle g\rangle=\langle h\rangle$, a contradiction. By symmetry in $g$ and $h$, every zero-sum subsequence $T\mid S$ also satisfies $\mathsf v_h(T)<np-1$.

It remains to consider $|\supp(S)|\ge3$. By Corollary~\ref{cor:regular-main}, we may assume that $S$ is not regular.
By Lemma~\ref{lem:4.0}, there exists a proper subgroup $H<G$ such that $S_H$ is regular over $H$, $|S_H|\ge|H|$, and $\Sigma(S_H)=H$.
Since $0\nmid S$, the bound $|S_H|\ge|H|$ implies $H\ne\{0\}$. We now distinguish two cases according to whether $H$ is cyclic.

\smallskip\noindent\emph{Case 1: $H$ is noncyclic.}
Each $g\in G=C_p\oplus C_{np}$ has a unique representation $g=(a,b)$ with $a\in C_p$ and $b\in C_{np}$.
Let $\pi:G\to C_{np}$ be the projection onto the second summand, so $\pi(a,b)=b$.
Since $\pi(H)\le C_{np}$ is cyclic, $\pi$ cannot be injective on $H$.
Since $|\ker\pi|=p$, we have $H\cap\ker\pi=\ker\pi=C_p\oplus\{0\}$. For every $(a,b)\in H$ we have $(a,0)\in H$, so $(0,b)\in H$.
Hence $H=C_p\oplus\pi(H)$, and $p\mid |\pi(H)|$ since $H$ is noncyclic. Thus
$H\cong C_p\oplus C_{mp}$ with $m\mid n$.
If $m=1$, then $\mathsf N(H)\le p^2\le|S_H|$ by \eqref{eq:N-elementary-rank-two}.
For $m\ge2$, Corollary~\ref{cor:regular-main} applies in $H$, since $p$ is the smallest prime divisor of $mp$, $S_H$ is regular over $H$, and $|S_H|\ge mp^2\ge2mp-1$.
In both cases, there exist $T\mid S_H$ and $h\mid S_H$ such that $\sigma(T)=0$ and $\mathsf v_h(T)=\mathsf v_h(S_H)\ge1$.
Since $S_H\mid S$ and $h\in H$, we have $T\mid S$ and $\mathsf v_h(T)=\mathsf v_h(S_H)=\mathsf v_h(S)$.

\smallskip\noindent\emph{Case 2: $H$ is cyclic.}
Let $m=|H|$. If $|S_H|\ge\mathsf N(H)$, then by the definition of $\mathsf N(H)$, there is a zero-sum subsequence $T\mid S_H$
with $\mathsf v_h(T)=\mathsf v_h(S_H)=\mathsf v_h(S)$ for some $h\mid S_H$.
Therefore, it suffices to consider $m\le|S_H|\le\mathsf N(H)-1$.
Let $S_1=SS_H^{-1}$, and let $\varphi:G\to G_1=G/H$ be the natural homomorphism.
If $G_1\cong C_p$, then $|H|=np$ and Lemma~\ref{lem:N-index-p-complete} applies.
If $G_1\cong C_q$ with $q>p$, the conclusion follows from Lemma~\ref{lem:N-cyclic-quotient}.

So we assume that $G_1$ is noncyclic. For every prime $q\ne p$,
the $q$-primary component of $G_1$ is a quotient of the cyclic $q$-primary
component of $G$. Hence the $p$-primary component of $G_1$ is noncyclic
and has rank two. It follows that $G_1/pG_1\cong G/(H+pG)$ has order
$p^2=|G/pG|$, so $H\le pG$. Since $pG$ is the subgroup of order $n$ in the
second direct summand, $1<m\mid n$ and $G_1\cong C_p\oplus C_{kp}$, where $k=n/m<n$.
If $\mathsf v_{\bar g}(\varphi(S_1))\ge\ord(\bar g)$ for some
$\bar g\in G_1\setminus\{0\}$ and
$|\varphi^{-1}(\bar g)\cap\supp(S_1)|\ge2$, then by Lemma~\ref{lem:N-heavy-fiber} there is a nonempty zero-sum subsequence $T\mid S$ with $\mathsf v_g(T)=\mathsf v_g(S)$ for some $g\mid S_1$.

So we assume that $|\varphi^{-1}(\bar g)\cap\supp(S_1)|=1$ for every $\bar g\mid\varphi(S_1)$ with $\mathsf v_{\bar g}(\varphi(S_1))\ge\ord(\bar g)$.
For each $g\mid S_1$, let $r_g=\ord(\varphi(g))$ and $s_g=\lfloor\mathsf v_g(S_1)/r_g\rfloor$.
Since $r_gs_g\le\mathsf v_g(S_1)$, we have $T_0=\prod_{g\in\supp(S_1)}(g^{r_g})^{s_g}\mid S_1$. Let $R=S_1T_0^{-1}$.
The sums $\sigma(g^{r_g})=r_g g$ form the sequence $U=\prod_{g\in\supp(S_1)}(r_g g)^{s_g}$.
For $g\mid S_1$ with $s_g>0$, we have $\varphi(r_g g)=r_g\varphi(g)=0$ and $1\le r_g\le\mathsf v_g(S_1)<\ord(g)$.
Thus $r_g g\in H\setminus\{0\}$ for every such $g$, and $0\nmid U$.
If $\mathsf v_{\bar g}(\varphi(S_1))\ge\ord(\bar g)$, then $\varphi^{-1}(\bar g)\cap\supp(S_1)=\{g\}$ for some $g$, and $\mathsf v_{\bar g}(\varphi(R))=\mathsf v_g(S_1)-r_g\lfloor\mathsf v_g(S_1)/r_g\rfloor<\ord(\bar g)$.
For each other $\bar g\mid\varphi(S_1)$, we have $\mathsf v_{\bar g}(\varphi(R))\le\mathsf v_{\bar g}(\varphi(S_1))<\ord(\bar g)$. Thus $\varphi(R)$ is weak-regular.

By Lemma~\ref{lem:N-kernel-blocks}, if $|U|\ge m-1$, then there exists a nonempty zero-sum subsequence $T\mid S$ with $\mathsf v_h(T)=\mathsf v_h(S)$ for some $h\mid S_H$.
For $|U|\le m-2$, the bound $r_g\le\exp(G_1)=kp$ for every $g\mid S_1$ implies
$|T_0|\le kp|U|\le(m-2)kp\le(m-1)kp-1$. Since $|S|\ge2np-2$, it follows that
\[
 |R|\ge2np-1-(\mathsf N(H)-1)-(m-1)kp.
\]
If $k\ge2$, then $m<n$ and Lemma~\ref{lem2.7} implies $|R|\ge2kp-1=\mathsf N(G_1)$, where the equality follows from the induction hypothesis for $G_1\cong C_p\oplus C_{kp}$.
For $k=1$, we have $m=n$ and $G_1\cong C_p^2$, so $|R|\ge\mathsf N(C_p^2)=\mathsf N(G_1)$ by Lemma~\ref{lem2.7}.

By the definition of $\mathsf N(G_1)$, there is $R_1\mid R$ with
$\sigma(\varphi(R_1))=0$ and
$\mathsf v_{\bar g}(\varphi(R_1))=\mathsf v_{\bar g}(\varphi(R))$
for some $\bar g\mid\varphi(R)$.
Thus $\mathsf v_h(R_1)=\mathsf v_h(R)$ for every $h\mid S_1$ with $\varphi(h)=\bar g$.
Since $\sigma(R_1),\sigma(T_0)\in H$ and $S_1=RT_0$, the subsequence $T_1=R_1T_0$
satisfies $\sigma(T_1)\in H$ and $\mathsf v_h(T_1)=\mathsf v_h(S_1)$ for every such $h$.
Since $\bar g\mid\varphi(R)$ and $R\mid S_1$, we have $\varphi^{-1}(\bar g)\cap\supp(S_1)\ne\emptyset$.
By Lemma~\ref{lem2.9}, since $\Sigma_0(S_H)=H$ and $\supp(S_H)\cap\supp(S_1)=\emptyset$, there exists $T\mid S$ with $\sigma(T)=0$ and $\mathsf v_h(T)=\mathsf v_h(S)\ge1$ for some $h\mid S_1$ with $\varphi(h)=\bar g$.
\end{proof}

Consider the sequences $S=g^{np-1}h^{np-1}$ in Theorem~\ref{thm1.5}, where $\ord(g)=\ord(h)=np$ and $\langle g\rangle\ne\langle h\rangle$.
We determine when an automorphism of $G$ sends one such sequence to another. Both $pg$ and $ph$ generate
$pG\cong C_n$, so there is a unique $t$ with $1\le t<n$ and $ph=tpg$.
We have $\gcd(t,n)=1$. Let $e_2=g$ and $e_1=h-tg$. Then $pe_1=0$
and $e_1\notin\langle g\rangle$. Since
$[G:\langle g\rangle]=p$, the elements $e_1,e_2$ form a basis of orders
$p,np$, respectively. Conversely, for any such basis and any integer $t$
coprime to $n$, the elements $g=e_2$ and $h=e_1+te_2$ generate distinct
cyclic subgroups of order $np$, since $h\notin\langle e_2\rangle$, $p\mid\ord(h)$, and $\ord(ph)=n$.
For another such sequence $S'=(g')^{np-1}(h')^{np-1}$, let $1\le t'<n$ satisfy $ph'=t'pg'$.
If an automorphism $\alpha$ of $G$ satisfies $\alpha(g)=g'$ and $\alpha(h)=h'$, then $ph'=\alpha(ph)=tpg'$, so $t'=t$.
Conversely, if $t'=t$, the assignments $g\mapsto g'$ and $h-tg\mapsto h'-tg'$ send the basis $(h-tg,g)$ to the basis $(h'-tg',g')$ and define an automorphism $\alpha$ with $\alpha(h)=h'$.
For the pair $(h,g)$ the parameter is $t^{-1}$ modulo $n$, since $pg=t^{-1}ph$.
Thus there is an automorphism $\alpha$ of $G$ with $\alpha(S)=S'$ if and only if $t'=t$ or $tt'\equiv1\pmod n$.

\section{Four families of cyclic groups}
For $n\ge3$, let $A_{n+1}=\lceil2\sqrt{n+1}\rceil$ and
\[
 B_{n+1}=\min\{a+b:a,b\in\N,\ ab\ge n+1,\ \gcd(a,n)=1\}.
\]
Gao et al.~\cite[Theorem~1.1(1)]{[GHHLYZ2022-2]} proved that
\begin{equation}\label{eq:cyclic-endpoints}
 2n-B_{n+1}+1\le\mathsf N(C_n)\le2n-A_{n+1}+1.
\end{equation}
Jiang and Yang \cite[Lemma~5.2]{[JY2024]} proved that $B_{n+1}=A_{n+1}+1$
for every $n\in T_1\cup T_2\cup T_3\cup T_4\cup T_5$, where
\begin{align*}
 T_1 &= \{(6t-3)^2-3:t\in\N\},\\
 T_2 &= \{2730(2730t^2-2835t+736):t\in\N\},\\
 T_3 &= \{2730(2730t^2-2625t+631):t\in\N\},\\
 T_4 &= \{(30030t+2192)^2-64:t=0,1,\ldots\},\\
 T_5 &= \{(30030t-2192)^2-64:t\in\N\}.
\end{align*}
For $n\in T_1$, $\mathsf N(C_n)=2n-A_{n+1}$ by
\cite[Proposition~5.2]{[GHHLYZ2022-2]} with $m=2t-1$.
For $n\in T_2\cup T_3\cup T_4\cup T_5$, we prove that $\mathsf N(C_n)=2n-A_{n+1}+1$.

\begin{thm}\label{thm5.3}
If $n\in T_2\cup T_3$, then $\mathsf N(C_n)=2n-B_{n+1}+2=2n-\lceil2\sqrt{n+1}\rceil+1$.
\end{thm}

\begin{proof}
By \eqref{eq:cyclic-endpoints}, it suffices to construct a weak-regular sequence $S$ of length $2n-A_{n+1}$ such that
$\mathsf v_g(T)<\mathsf v_g(S)$ for every nonempty zero-sum subsequence
$T\mid S$ and every $g\mid S$.
Let $C_n=\langle e\rangle$.

For $n\in T_2$, let $d=1365(2t-1)-52$ and $x=(12n-d)/13$.
Then $n=d(d-1)-26$, $13\mid d$, and $d$ is odd.
Since $d\ge1313$, we have $(d-1)^2<n+1<(d-\tfrac12)^2$ and
hence $A_{n+1}=2d-1$. Reducing
$13x=12n-d$ and $(d-1)x=(12d/13-1)n-2$ modulo $n$, we obtain
\[
 13x\equiv-d\pmod n,\qquad (d-1)x\equiv-2\pmod n.
\]
Thus $\gcd(x,n)\mid\gcd(d,2)=1$. Since $1<x<n$, the sequence $S=e^{n-d}(xe)^{n-d+1}$
is weak-regular and has length $2n-A_{n+1}$.
Assume to the contrary that a zero-sum subsequence $T\mid S$ satisfies $\mathsf v_g(T)=\mathsf v_g(S)$ for some $g\in\{e,xe\}$. Since $\sigma(T)=0$, we have $\mathsf v_e(T)+x\mathsf v_{xe}(T)\equiv0\pmod n$.
If $\mathsf v_{xe}(T)=n-d+1$, then $\mathsf v_e(T)\equiv(d-1)x\equiv-2\pmod n$, so $\mathsf v_e(T)\ge n-2>n-d$, a contradiction.
So we assume that $\mathsf v_e(T)=n-d$. Since $13x\equiv-d\pmod n$ and $\gcd(x,n)=1$,
we have $\mathsf v_{xe}(T)\equiv-13\pmod n$, so $\mathsf v_{xe}(T)\ge n-13>n-d+1$, a contradiction.

So we assume that $n\in T_3$. Let $d=1365(2t-1)+53$ and $x=(n-d)/2$.
Then $n=d(d-1)-26$, $d$ is even, and $d\equiv1\pmod{13}$.
Since $d\ge1418$, we have $(d-1)^2<n+1<(d-\tfrac12)^2$ and hence $A_{n+1}=2d-1$.
Reducing $2x=n-d$ and $(d-1)x=((d-2)/2)n-13$ modulo $n$, we obtain
\[
 2x\equiv-d\pmod n,\qquad (d-1)x\equiv-13\pmod n.
\]
Thus $\gcd(x,n)\mid\gcd(d,13)=1$. Since $1<x<n$, the sequence $S=e^{n-d}(xe)^{n-d+1}$
is weak-regular with $|S|=2n-A_{n+1}$.
Assume to the contrary that a zero-sum subsequence $T\mid S$ satisfies $\mathsf v_g(T)=\mathsf v_g(S)$ for some $g\in\{e,xe\}$. Since $\sigma(T)=0$, we have $\mathsf v_e(T)+x\mathsf v_{xe}(T)\equiv0\pmod n$.
If $\mathsf v_{xe}(T)=n-d+1$, then $\mathsf v_e(T)\equiv(d-1)x\equiv-13\pmod n$, so $\mathsf v_e(T)\ge n-13>n-d$, a contradiction.
So we assume that $\mathsf v_e(T)=n-d$. Since $2x\equiv-d\pmod n$ and $\gcd(x,n)=1$,
we have $\mathsf v_{xe}(T)\equiv-2\pmod n$, so $\mathsf v_{xe}(T)\ge n-2>n-d+1$, a contradiction.
Thus $\mathsf N(C_n)\ge |S|+1=2n-A_{n+1}+1$ for every $n\in T_2\cup T_3$.
\end{proof}

\begin{thm}\label{thm:cyclic-T45}
Let $n=d^2-64$, where $d=30030t+2192$ for an integer $t\ge0$
or $d=30030t-2192$ for an integer $t\ge1$.
Then $\mathsf N(C_n)=2n-B_{n+1}+2=2n-2d+1$.
\end{thm}

\begin{proof}
The two formulas for $d$ correspond to $T_4$ and $T_5$, respectively,
so $B_{n+1}=A_{n+1}+1$ by \cite[Lemma~5.2]{[JY2024]}.
Since $d\ge2192$, we have $(d-\tfrac12)^2<n+1<d^2$, and hence
$A_{n+1}=2d$. Thus $\mathsf N(C_n)\le 2n-2d+1$ by \eqref{eq:cyclic-endpoints}.

To construct a sequence of length $2n-2d$, let $a=d+7$ and $b=d-7$,
so that $ab=n+15$. Since $d\equiv\pm2\pmod{15}$,
let $(u,v)=(3,5)$ if $d\equiv2\pmod{15}$ and $(u,v)=(5,3)$ otherwise.
In both cases $uv=15$, $\gcd(u,b)=1$, and $\gcd(a,v)=1$.
Since $\gcd(u,b)=1$, there exist integers $i,j$ with $iu+jb=1$. Let $0\le x<n$ satisfy
$x\equiv-ia-jv\pmod n$. Since $ab\equiv uv\pmod n$,
\[
 ux\equiv-a(iu+jb)=-a\pmod n,\qquad
 bx\equiv-v(iu+jb)=-v\pmod n.
\]
Hence $\gcd(x,n)\mid\gcd(a,v)=1$.
Also $x\ne1$, since $0<a+u<n$.

Let $C_n=\langle e\rangle$ and $S=e^{n-a}(xe)^{n-b}$. We show that every zero-sum subsequence $T\mid S$ satisfies $\mathsf v_e(T)<n-a$ and $\mathsf v_{xe}(T)<n-b$.
The elements $e$ and $xe$ are distinct and have order $n$, so $S$ is
weak-regular and $|S|=2n-2d$. Assume to the contrary that a zero-sum subsequence $T\mid S$ satisfies $\mathsf v_g(T)=\mathsf v_g(S)$ for some $g\in\{e,xe\}$. Since $\sigma(T)=0$, we have $\mathsf v_e(T)+x\mathsf v_{xe}(T)\equiv0\pmod n$.
If $\mathsf v_{xe}(T)=n-b$, then $\mathsf v_e(T)\equiv bx\equiv-v\pmod n$, so $\mathsf v_e(T)\ge n-v>n-a$, a contradiction.
So we assume that $\mathsf v_e(T)=n-a$. Since $ux\equiv-a\pmod n$ and $\gcd(x,n)=1$,
we have $\mathsf v_{xe}(T)\equiv-u\pmod n$, so $\mathsf v_{xe}(T)\ge n-u>n-b$, a contradiction. Thus $\mathsf N(C_n)\ge|S|+1=2n-2d+1$.
\end{proof}

\begin{cor}\label{cor:cyclic-finite-range}
For every integer $3\le n\le 10^8$,
\[
 \mathsf N(C_n)=
 \begin{cases}
  2n-\lceil2\sqrt{n+1}\rceil, & n\in T_1,\\
  2n-\lceil2\sqrt{n+1}\rceil+1, & n\notin T_1,
 \end{cases}
\]
where $T_1=\{(6t-3)^2-3:t\in\N\}$.
\end{cor}

\begin{proof}
For $n\in T_1$ the formula follows from \cite[Proposition~5.2]{[GHHLYZ2022-2]}.
By \cite[Lemma~5.1]{[JY2024]}, the nine values of $n\notin T_1$ in $[3,10^8]$
for which $B_{n+1}=A_{n+1}+1$ are $1722630$, $2009280$, $4804800$, $16341780$,
$17201730$, $45866730$, $47299980$, $90297480$, and $92304030$.
All of them belong to $T_2\cup T_3\cup T_4$, so $\mathsf N(C_n)=2n-A_{n+1}+1$
by Theorems~\ref{thm5.3} and~\ref{thm:cyclic-T45}.
For every other $n\notin T_1$ in this range, $B_{n+1}=A_{n+1}$ by
\cite[Lemma~5.1]{[JY2024]}, so \eqref{eq:cyclic-endpoints} implies
$\mathsf N(C_n)=2n-A_{n+1}+1$.
\end{proof}

\section*{Acknowledgements}
This paper is based on parts of Chapters~4 and~5 of the author's
doctoral dissertation, completed in May 2025, and has not previously
appeared in any journal.
This work was supported in part by the National Natural Science Foundation
of China (grant number 12171255).

\bibliographystyle{amsplain}

\begin{thebibliography}{99}

\bibitem{[EH1964]} P. Erd\H{o}s and H. Heilbronn,
\emph{On the addition of residue classes mod $p$},
Acta Arith. \textbf{9} (1964), 149--159.

\bibitem{[G2000]} W. Gao,
\emph{On Davenport's constant of finite abelian groups with rank three},
Discrete Math. \textbf{222} (2000), 111--124.

\bibitem{[GG2006]} W. Gao and A. Geroldinger,
\emph{Zero-sum problems in finite abelian groups: a survey},
Expo. Math. \textbf{24} (2006), no.~4, 337--369.

\bibitem{[GGWZ2026]} W. Gao, A. Geroldinger, G. Wang and X. Zeng,
\emph{A disproof of the van Emde Boas--Kruyswijk conjecture on the Davenport constant},
arXiv:2611.24775v2 (2026).

\bibitem{[GHQQZ2015]} W. Gao, D. Han, G. Qian, Y. Qu and H. Zhang,
\emph{On additive bases II}, Acta Arith. \textbf{168} (2015), 247--267.

\bibitem{[GHHLYZ2022-1]} W. Gao, S. Hong, W. Hui, X. Li, Q. Yin and P. Zhao,
\emph{Representation of zero-sum invariants by sets of zero-sum
sequences over a finite abelian group},
Period. Math. Hung. \textbf{85} (2022), 52--71.

\bibitem{[GHHLYZ2022-2]} W. Gao, S. Hong, W. Hui, X. Li, Q. Yin and P. Zhao,
\emph{Representation of zero-sum invariants by sets of zero-sum
sequences over a finite abelian group II},
J. Number Theory \textbf{241} (2022), 738--760.

\bibitem{[GLPW2018]} W. Gao, Y. Li, J. Peng and G. Wang,
\emph{A unifying look at zero-sum invariants},
Int. J. Number Theory \textbf{14} (2018), 705--711.

\bibitem{[GLQW2022]} W. Gao, Y. Li, Y. Qu and Q. Wang,
\emph{Additive bases of finite abelian groups of rank 2},
Integers \textbf{22} (2022), Paper No. A37.

\bibitem{[GQZ2020]} W. Gao, Y. Qu and H. Zhang,
\emph{On additive bases III}, Acta Arith. \textbf{193} (2020), 293--308.

\bibitem{[GG2009]} A. Geroldinger and D. J. Grynkiewicz,
\emph{On the arithmetic of Krull monoids with finite Davenport constant},
J. Algebra \textbf{321} (2009), 1256--1284.

\bibitem{[GGZ2026]} A. Geroldinger, D. J. Grynkiewicz and Q. Zhong,
\emph{Combinatorial Factorization Theory: Factorization Theory in Krull
Monoids via Additive Combinatorics}, Mathematical Surveys and Monographs,
Vol.~296, American Mathematical Society, Providence, RI, 2026.

\bibitem{[GH2006]} A. Geroldinger and F. Halter-Koch,
\emph{Non-Unique Factorizations. Algebraic, Combinatorial and Analytic Theory},
Pure and Applied Mathematics, Vol.~278, Chapman \& Hall/CRC,
Boca Raton, FL, 2006.

\bibitem{[GS2021]} A. Geroldinger and W. A. Schmid,
\emph{A characterization of Krull monoids for which sets of lengths are (almost) arithmetical progressions},
Rev. Mat. Iberoam. \textbf{37} (2021), 293--316.

\bibitem{[GS1992]} A. Geroldinger and R. Schneider,
\emph{On Davenport's constant},
J. Combin. Theory Ser. A \textbf{61} (1992), no.~1, 147--152.

\bibitem{[Hall1935]} P. Hall, \emph{On representatives of subsets},
J. London Math. Soc. \textbf{10} (1935), 26--30.

\bibitem{[HZ2026]} D. Han and H. Zhang,
\emph{Davenport constants of finite abelian groups},
arXiv:2610.15102v1 (2026).

\bibitem{[JY2024]} X. Jiang and W. Yang,
\emph{A zero-sum problem related to the max gap of the unit group of the residue class ring},
J. Number Theory \textbf{260} (2024), 198--211.

\bibitem{[L2026]} C. Liu,
\emph{Davenport constants with growing binary rank},
arXiv:2610.38447v1 (2026).

\bibitem{[L2020]} C. Liu,
\emph{On the lower bounds of Davenport constant},
J. Combin. Theory Ser. A \textbf{171} (2020), Article 105162, 15 pp.

\bibitem{[N1996]} M. B. Nathanson,
\emph{Additive Number Theory: Inverse Problems and the Geometry of Sumsets},
Graduate Texts in Mathematics, Vol.~165, Springer, New York, 1996.

\bibitem{[P1987I]} C. Peng,
\emph{Addition theorems in elementary abelian groups I},
J. Number Theory \textbf{27} (1987), 46--57.

\bibitem{[P1987II]} C. Peng,
\emph{Addition theorems in elementary abelian groups II},
J. Number Theory \textbf{27} (1987), 58--62.

\bibitem{[QH2017]} Y. Qu and D. Han,
\emph{Additive bases of $C_p\oplus C_{p^n}$},
Int. J. Number Theory \textbf{13} (2017), 2453--2459.

\bibitem{[QL2022]} Y. Qu and Y. Li,
\emph{Additive bases of $C_3\oplus C_{3q}$},
Colloq. Math. \textbf{169} (2022), 189--195.

\bibitem{[QLWZ2025]} Y. Qu, Y. Li, Q. Wang and X. Zhao,
\emph{A note on additive bases of abelian groups of rank 2},
Graphs Combin. \textbf{41} (2025), Paper No. 24.

\bibitem{[R1963]} K. Rogers,
\emph{A combinatorial problem in Abelian groups},
Proc. Cambridge Philos. Soc. \textbf{59} (1963), 559--562.

\bibitem{[S2011]} W. A. Schmid,
\emph{The inverse problem associated to the Davenport constant for
$C_2\oplus C_2\oplus C_{2n}$, and applications to the arithmetical
characterization of class groups}, Electron. J. Combin. \textbf{18}
(2011), no.~1, Paper P33, 42 pp.

\bibitem{[W2026]} G. Wang,
\emph{Disproving conjectures on the Davenport and Narkiewicz constants},
arXiv:2609.29878v5 (2026).

\bibitem{[W2025]} G. Wang,
\emph{The universal zero-sum invariant and weighted zero-sum for infinite abelian groups},
Comm. Algebra \textbf{53} (2025), no.~4, 1581--1599.

\bibitem{[W2026II]} G. Wang,
\emph{The universal zero-sum invariant and weighted zero-sum for infinite abelian groups II},
arXiv:2607.02132v1 (2026).

\end{thebibliography}

\end{document}